\documentclass{amsart}
\usepackage{geometry,amssymb,enumerate}

\numberwithin{equation}{section}

\newtheorem{thm}{Theorem}[section]
\newtheorem{lm}[thm]{Lemma}
\newtheorem{prop}[thm]{Proposition}

\theoremstyle{definition}
\newtheorem{dfn}[thm]{Definition}
\newtheorem{rmq}[thm]{Remark}

\newtheorem*{plan}{Plan of the paper}
\newtheorem*{ack}{Acknowledgments}

\usepackage{color}

\usepackage[colorlinks=true,urlcolor=blue,
citecolor=red,linkcolor=blue,linktocpage,pdfpagelabels,bookmarksnumbered,bookmarksopen
]{hyperref}
\newcommand{\Je}{\mathbf J^{\rm E}}
\newcommand{\Jm}{\mathbf J^{\rm M}}
\newcommand{\Hu}{H^1(\Omega\mathbin{;}{\rm curl})}
\newcommand{\Huz}{H^1_0(\Omega\mathbin{;}{\rm curl})}
\newcommand{\Ld}{{L^2}}
\newcommand{\virg}{\mathbin{,}}

\def\XXint#1#2#3{{\setbox0=\hbox{$#1{#2#3}{\int}$} \vcenter{\vspace{-1pt}\hbox{$#2#3$}}\kern-.5\wd0}}
\def\Xint#1{\mathchoice {\XXint\displaystyle\textstyle{#1}}{\XXint\textstyle\scriptstyle{#1}}{\XXint\scriptstyle\scriptscriptstyle{#1}}{\XXint\scriptscriptstyle\scriptscriptstyle{#1}}\!\int}
\def\intmed{\Xint{-}}

\def\step#1#2{\par\noindent{\underline{\it Step~#1.}}\emph{ #2}\\}

\newcommand{\vertiii}[1]{{\left\vert\kern-0.25ex\left\vert\kern-0.25ex\left\vert #1 
    \right\vert\kern-0.25ex\right\vert\kern-0.25ex\right\vert}}

\author[Francini]{Elisa Francini}
\author[Franzina]{Giovanni Franzina}
\author[Vessella]{Sergio Vessella}

\address[G. Franzina]{
Unit\`a di Ricerca  in{\rm\tiny d}am di Firenze \mbox{%
\raisebox{.8ex}{\rm c}%
\kern-.175em\raisebox{.2ex}{/}%
\kern-.18em\raisebox{-.2ex}{\rm o}%
} d{\rm\tiny i}m{\rm\tiny a}i ``Ulisse Dini'' 
\newline\indent 
Universit\`a degli Studi di Firenze
\newline\indent 
Viale Morgagni 67/A, 50134 Firenze, Italy}
\email{franzina@math.unifi.it}
\address[E. Francini, S. Vessella]{
d{\rm\tiny i}m{\rm\tiny a}i ``Ulisse Dini'' 
\newline\indent 
Universit\`a degli Studi di Firenze
\newline\indent 
Viale Morgagni 67/A, 50134 Firenze, Italy}
\email{francini@math.unifi.it,vessella@unifi.it}


\title{Existence and regularity for eddy current system with non-smooth conductivity}

\begin{document}
\maketitle

\begin{abstract}
We discuss the well-posedness of the ``transient eddy current'' magneto-quasistatic approximation of Maxwell's initial value problem
with bounded and measurable conductivity, with sources, on a domain. We prove existence and uniqueness of weak solutions, and we provide
global H\"older estimates for the magnetic part.
\end{abstract}

\section{Introduction}

Let $\Omega$ be a bounded $C^{1,1}$ domain in $\mathbb R^3$ (see Section~\ref{regdom} for definitions), and let $n$ denote
the outward unit normal to its boundary. We consider 
electromagnetic signals throughout a medium, filling the region $\Omega$, with {\em magnetic permeability} being given by
a Lipschitz continuous scalar function $\mu$
and 
{\em electric conductivity} being described by a bounded measurable function $\sigma$
taking  values in the real symmetric  $3\times3$ matrices.  We will assume
the validity of the conditions
\begin{subequations}\label{structure}
\renewcommand{\theequation}{\theparentequation \roman{equation}}
\begin{align}
\label{1.2i}
\Lambda^{-1} \le \mu \le \max\{ \mu\virg|\nabla\mu|\}\le \Lambda \,,& \qquad \text{a.e.\ in $\Omega$,}\\
\label{1.2ii}
\Lambda^{-1} |\eta|^2\le \sigma\eta\cdot\eta \le \Lambda |\eta|^2\,,&  \qquad \text{for all $\eta\in\mathbb R^3$, a.e.\ in $\Omega$,}
\end{align}
\end{subequations}
for an appropriate constant $\Lambda\ge1$.

Given $T>0$,
$\mathbf H_0\in \Ld(\Omega\mathbin{;}\mathbb R^3)$, 
 $\mathbf G\in L^2(0,T\mathbin{;} \Hu)$, with $\partial_t\mathbf G\in L^2(0,T\mathbin{;}L^2(\Omega\mathbin{;}\mathbb R^3))$,
and $\Je,\Jm\in L^2(0,T;L^2(\Omega\mathbin{;}\mathbb R^3))$, we consider weak
solutions $(\mathbf E\virg\mathbf H)\in L^2(0,T\mathbin{;}\Hu\times\Huz)$, with
$\partial_t\mathbf H\in L^2(0,T\mathbin{;}L^2(\Omega\mathbin{;}\mathbb R^3))$  (see Section~\ref{s:2} for definitions), of the initial value problem
\begin{equation}
\label{1.1}
\begin{cases}
 \nabla\times \mathbf H -\sigma \mathbf E = \Je \,, &\quad \text{in $\Omega\times(0,T)$,}\\
 \nabla \times\mathbf E + \mu \partial_t \mathbf H = \Jm\,,& \quad \text{in $\Omega\times(0,T)$,}\\
 \mathbf H\times n = \mathbf G\times n\,, & \quad \text{on $\partial\Omega\times(0,T)$,}\\
 \mathbf H=\mathbf H_0\,, & \quad \text{in $\Omega\times\{0\}$,}
\end{cases}
\end{equation}
under the assumption that
\begin{equation}
\label{in:comp}
	\nabla\cdot \left( \mu \mathbf G-\mu \mathbf H_0-\int_0^t\Jm\,ds\right)=0\,, \qquad \text{in $\Omega\times(0,T)$.}
\end{equation}
The meaning of \eqref{1.1} and of \eqref{in:comp} will be understood in a suitable weak sense in Section~\ref{s:2}.

Formally, the so-called eddy current system \eqref{1.1} is obtained
from Maxwell's equations when neglecting displacement currents and is equivalent to the parabolic system
\begin{equation}
\label{parabolic}
	\mu\partial_t\mathbf H + \nabla\times \Big( \sigma^{-1} \nabla\times \mathbf H\Big)=\nabla\times (\sigma^{-1}\Je)+\Jm\,,\quad \text{in }\Omega\times(0,T)\,,
\end{equation}
with the conditions $\mathbf H\times n = \mathbf G\times n$ on $\partial\Omega \times(0,T)$ and $\mathbf H = \mathbf H_0$ in $\Omega\times\{0\}$,
provided that
\[
	\mathbf E = \sigma^{-1} \Big( \nabla\times\mathbf H-\Je\Big)\,,\quad\text{in }\Omega\times(0,T)\,.
\]
To make an example, if $\sigma$ is constant and $\Je=\Jm=0$, then \eqref{parabolic} reads as
\[
\mu\sigma \partial_t\mathbf H + \nabla\times\nabla\times \mathbf H = 0\,,\quad\text{in }\Omega\times(0,T)\,,
\]
and $\nabla\times\nabla\times \mathbf H  = \nabla(\nabla\cdot \mathbf H) - \Delta \mathbf H$, where the Laplace operator is understood componentwise. 
Hence, in this case the problem is equivalent to the heat equation for the Hodge-Laplacian on vector fields, and
the components of divergence-free solutions solve the classical heat equation (up to a weight).

Our interest in this
 {\em parabolic} magneto-quasistatic
approximation of 
the laws of classical electromagnetism with possibly {\em discontinuous} electric conductivity tensor comes from {\em diffusive} models in applied 
seismo-electromagnetic studies~\cite{palangio,yamazaki}. In geophysics, the importance of modelling
{\em slowly varying} electromagnetic fields throughout the {\em stratified} lithosphere is due to the possibility
that some of them may be generated by co-seismic subsurface electric currents, and hence have some r\^ole in the seismic
percursor signal recognition.
For a very general survey on eddy currents with discontinuous conductivity and related numerics, with applications
to advanced medical diagnostics, the interested reader is referred instead to the nice treatise~\cite{AV}, where
inverse problems are also considered. We refer to~\cite{ACV} for issues related to
the source identification from boundary EM measurement.

The main results of this manuscript concern some qualitative properties of weak solutions of \eqref{1.1}, i.e., their existence and uniqueness, as well as the H\"older continuity of their magnetic part. For expositional purposes, we limit ourselves to the case of homogeneous boundary conditions, which causes no restriction (see Section~\ref{ss:ws}).

In Theorem~\ref{teo1} (see Section~\ref{s:ex}), we prove the well-posedness of \eqref{1.1}; for, we make use
of Galerkin's method and of the Hilbert basis that we manufacture in Section~\ref{ss:2.1} by solving an auxiliary problem of spectral type. This 
special system of vector fields has the expedient feature of being independent of the conductivity stratification, at variance with 
the natural basis for the associated parabolic problem. Existence and uniqueness results are available in the literature for problems
similar to \eqref{1.1}; for example,
in the time-harmonic regime the issue of well-posedness was addressed in \cite{H}, and 
in~\cite{ABN} (where it is also proved to be a good approximation of the complete set of Maxwell's equations), and
the time-harmonic variant of \eqref{1.1} is also dealt with in the  more recent paper~\cite{ACCGV}, providing
existence and uniqueness results and asymptotic expansions in terms of the size of the conductor in this context,
whereas in~\cite{AH} the well-posedness
of the variant of this problem focused on the electric field is discussed using a different approach, in the time domain, with applications to the asymptotic behaviour of solutions in the non-conductive limit.

In Theorem~\ref{teoreg2} (see Section~\ref{s:reg}), inspired by the work~\cite{A} on Maxwell's system, we prove H\"older continuity estimates for the magnetic field, valid up to the boundary. In the literature, we could not find either global or local estimate of this kind; we refer to the paper~\cite{Co} for some related result.

\begin{plan}
In Section~\ref{s:2} we make precise assumptions
on the domain and on the structure of the problem, we introduce the reader to some useful functional-analytic tools, 
we state some Helmoltz-type decompositions (proved in Appendix), and we define the weak solutions of the eddy current system \eqref{1.1}.
In Section~\ref{s:ex} we prove existence and uniqueness of weak solutions $(\mathbf E\virg\mathbf H)$, 
and in Section~\ref{s:reg} we provide global {\em a-priori} H\"older estimates on the magnetic field $\mathbf H$. 
\end{plan}

\begin{ack}
This research is supported by the {\scshape miur}-{\scshape foe-in}{\scriptsize d}{\scshape am} 2014 grant ``Strategic Initiatives for the Environment and Security - SIES''.
\end{ack}

\section{Technical Tools}\label{s:2}
We recall that the tangential trace, defined by $\phi\times n$ for all $\phi\in C^1(\overline\Omega;\mathbb R^3)$, extends to a bounded operator
from the Hilbert  space $\Hu$, consisting of all 
vector fields in $L^2(\Omega\mathbin{;}\mathbb R^3)$ whose (distributional) curl is also in $L^2(\Omega\mathbin{;}\mathbb R^3)$, endowed with the scalar product
\begin{equation}
\label{2.1}
	(\varphi\virg \psi)_{\Hu} = (\varphi\virg\psi)_{L^2(\Omega;\mathbb R^3)} + (\nabla\times \varphi\virg \nabla\times \psi)_{L^2(\Omega;\mathbb R^3)}\,,
\end{equation}
to the dual space $H^{-\frac{1}{2}}(\partial\Omega\mathbin{;}\mathbb R^3)$ of  $H^{\frac{1}{2}}(\partial\Omega\mathbin{;}\mathbb R^3)$ (see, e.g.,~\cite{DL}). 
Indeed, the Green-type formula
\begin{equation}
\label{2.2}
\int_\Omega \varphi\cdot\nabla\times\psi\,dx - \int_\Omega \psi\cdot\nabla\times\varphi\,dx = -\int_{\partial\Omega} \varphi\cdot(\psi\times n)\,dS
\end{equation}
holds for all $(\varphi,\psi){\in} (C^1(\overline\Omega\mathbin{;}\mathbb R^3))^2$. Moreover,
given $\psi\in \Hu$, by Sobolev extension and trace theorems, 
the left hand-side of \eqref{2.2} defines a bounded linear operator on $H^{\frac{1}{2}}(\partial\Omega\mathbin{;}\mathbb R^3)$
and
for every $\varphi\in \Hu$ formula \eqref{2.2} holds valid provided that the right hand-side is understood in a suitable weak sense, replacing the boundary integral with a
duality pairing.

 The closed subspace $\Huz$ of all $\psi\in \Hu$ for which, in the previous weak sense, we have $ \psi\times n=0$ on $\partial\Omega$ is
also a Hilbert space with respect to \eqref{2.1}.

Throughout the paper, the spaces of $L^2$ scalar-valued, vector-valued, and tensor-valued functions will be denoted by $L^2(\Omega)$,
$L^2(\Omega\mathbin{;}\mathbb R^3)$, $L^2(\Omega\mathbin{;}\mathbb R^{3\times3})$, respectively. For the sake of readability, we shall denote 
by $(\cdot\virg\cdot)_{L^2}$ and $\|\cdot\|_{L^2}$ the scalar product and the norm in all these spaces.

\subsection{Regularity of the domain}\label{regdom}
An open set $\Omega$ is said to satisfy the uniform two-sided ball condition with radius $r$ if for every
$z\in\partial\Omega$ there exist a ball $B_r(x)$ contained in $\Omega$ and a ball $B_r(y)$ contained in its complement with $z$ belonging to the closure of both $B_r(x)$
and of $B_r(y)$. 
If that is the case and we assume, in addition, that $\partial\Omega = \partial(\overline\Omega)$, then
$\Omega$ is a locally $C^{1,1}$-domain, i.e.,
for every $z\in \partial\Omega$ there exist two positive constants $\rho_0, L_0>0$,
and a rigid change of coordinates in $\mathbb R^3$, under which $z=0$ and
\[
	\Omega\cap B_{\rho_0}(0)  = \{y\in B_{\rho_0}(0) \colon y_3>\varphi(y_1,y_2)\}\,,
\]
for some $ C^{1,1}$ function $\varphi$  on $B_{\rho_0}'=\{(y_1,y_2)\in \mathbb R^2\colon y_1^2+y_2^2<\rho_0^2\}$, with $	\varphi(0) {=}|\nabla \varphi(0)|{=}0$, such that
\[
\|\varphi\|_{L^\infty(B_{\rho_0}')} + \rho_0\|\nabla\varphi\|_{L^\infty(B_{\rho_0}')} + \rho_0^2 \mathop{\rm Lip}(\nabla\varphi\mathbin{;}B_{\rho_0}')\le L_0 \rho_0\,,
\]
where 
\[
\mathop{\rm Lip}(\nabla\varphi\mathbin{;}B_{\rho_0}')  = \sup_{\substack{x,y\in B_{\rho_{\scriptsize0}}' \\ y\neq z}}\frac{|\nabla \varphi(x)-\nabla\varphi(y)|}{|x-y|}\,.
\]
If $\Omega$ is bounded and the property described above holds with constants $\rho_0,L_0$ independent of $z$, then we say that
$\Omega$ is  of class $C^{1,1}$ with constants $\rho_0,L_0$. In that case, it is easily seen that
$\Omega$ satisfies the uniform two-sided ball condition with radius $r$, provided that
$r<\min\{1\virg L_0^{-1}\}\rho_0$.
\vskip.2cm

Throughout this paper we shall always assume the following condition to be in force:
\begin{equation}
\label{hpOmega}
\text{$\Omega$ is bounded, with uniform two-sided ball condition with radius $r$, and
 $\partial\Omega = \partial(\overline\Omega)$.}
\end{equation}
We observe that \eqref{hpOmega}
implies that $\Omega$ is of class $C^{1,1}$ with appropriate constants $\rho_0,L_0$, satisfying $L_0 r<\rho_0$ (see~\cite[Corollary 3.14]{ABMMZ}), 
and we shall assume that $\rho_0=1$ with no loss of generality.

\subsection{Gaffney inequality}
The following result is proved in \cite{F} in the case of domains with smooth boundaries but its validity is also well known on open sets satisfying assumption (2.3) (see, e.g.,~\cite{Co0}).
\begin{lm}[Gaffney inequality]\label{lm:gaffney-aniso}
Let $ \psi\in L^2(\Omega\mathbin{;}\mathbb R^3)$, with
$\nabla\cdot\psi\in L^2(\Omega)$ and $\nabla\times\psi\in L^2(\Omega\mathbin{;}\mathbb R^3)$. If
either $\psi\times n=0$ in $H^{-\frac{1}{2}}(\partial\Omega\mathbin{;}\mathbb R^3)$ or $\psi\cdot n=0$ in $H^{-\frac{1}{2}}(\partial\Omega)$, then $\psi\in H^1(\Omega\mathbin{;}\mathbb R^3)$. 
Moreover,
\begin{equation}
\label{gaffney-aniso}
	\int_\Omega (\nabla\cdot\psi)^2\,dx+\int_\Omega|\nabla\times\psi|^2\,dx+\int_\Omega|\psi|^2\,dx\ge C \int_\Omega|\nabla\psi|^2\,dx\,,
\end{equation}
where the constant $C$ depends on $r$, only. 
\end{lm}
For every $\mu\in L^\infty(\Omega)$, we set
\begin{equation}
\label{XmuYmu}
	X_\mu= \left\{\psi\in L^2(\Omega\mathbin{;}\mathbb R^3)\colon \int_\Omega\mu \psi\cdot\nabla u\,dx = 0 \,,\ \text{for all $u\in H^1_0(\Omega)$}\right\}\,,
	\quad 
	Y_\mu = \Huz\cap X_\mu \,.
\end{equation}
If $\mu=1$ then, to shorten the notation, we write $X$, $Y$ instead of $X_\mu$, $Y_\mu$.

Clearly  if \eqref{1.2i} holds then $X_\mu$ is a Hilbert space with respect to the $\Ld(\mu)$-scalar product, i.e.
\begin{equation}
\label{weight1}
	\left( \varphi\virg \psi\right)_{X_\mu} := \int_\Omega \mu \varphi\cdot\psi\,dx\,, \quad \text{for all } \varphi\,, \psi\in X_\mu\,.
\end{equation}
The space $Y_\mu$ is closed in $\Huz$ with respect to the topology induced by \eqref{weight1} which in fact is the standard topology of
$\Ld(\Omega\mathbin{;}\mathbb R^3)$, as $\mu\in L^\infty(\Omega)$.
It is straightforward to deduce the following result from Lemma~\ref{lm:gaffney-aniso}.
\begin{lm}\label{mettidiv}
Let $\Omega$ satisfy the uniform interior and exterior ball condition with radius $r$ and let $\mu$ satisfy \eqref{1.2i}. Then, every $\psi\in Y_\mu$ belongs
to the Sobolev space $H^1(\Omega\mathbin{;}\mathbb R^3)$ and we have
\[ 
\int_\Omega |\nabla\psi|^2\,dx\le C\left( \int_\Omega |\psi|^2\,dx+\int_\Omega |\nabla\times\psi|^2\,dx\right)\,,
\]
for a suitable constant $C$, depending only on $\Lambda$ and $r$.
\end{lm}

\begin{rmq}\label{rmq:norm}
By Lemma~\ref{mettidiv}, if \eqref{1.2i} holds then the norm
\[
	\|\psi\|_{Y_\mu}:=\left( \int_\Omega\mu\, |\psi|^2\,dx+\int_\Omega\mu\, |\nabla\times \psi|^2\,dx  \right)^\frac{1}{2}\,,
\]
is equivalent to that induced on $Y_\mu$ by $H^{1}(\Omega\mathbin{;}\mathbb R^3)$.
\end{rmq}

\begin{rmq}
\label{rmq:norm2}
By Remark~\ref{rmq:norm},  the compactness of the embedding of $H^1(\Omega\mathbin{;}\mathbb R^3)$ 
into $L^2(\Omega\mathbin{;}\mathbb R^3)$ implies that the embedding of $Y_\mu$ into $X_\mu$ is compact
if condition \eqref{1.2i} holds.
\end{rmq}

\subsection{Helmoltz decomposition}
\label{ss:hd}
We shall make use of the following  Helmoltz-type decompositions. The interested reader may find in the appendix their proofs, that are however
standard.
\begin{lm}\label{HelmE}
	Let $\mathbf F\in L^2(\Omega\mathbin{;}\mathbb R^3)$. Then there exist $u\in H^1(\Omega) $ and $\eta\in L^2(\Omega\mathbin{;}\mathbb R^3)$ such that
\begin{subequations}\label{HelmEeq}
\begin{align}
\label{HelmE1}
	& \mathbf F = \nabla u+\eta \,,& \\
\label{HelmE2}
	& \int_\Omega \eta\cdot\nabla v\,dx=0 \,, \quad \text{for all }v\in H^1(\Omega),&		\\
\label{HelmE2.5}
	& \max\Big\{ \|\nabla u\|_{\Ld} \virg \| \eta\|_{\Ld} \Big\} \le \|\mathbf F\|_{\Ld}\,.
\end{align}
\end{subequations}
If in addition $\mathbf F\in \Hu$, then $\eta\in H^1(\Omega\mathbin{;}\mathbb R^3)$ and $\|\nabla\eta\|_{\Ld}=\|\nabla\times\mathbf F\|_{\Ld}$.
\end{lm}

\begin{lm}\label{lm:hh}
Let $\mu$ satisfy \eqref{1.2i}. Given $\mathbf F\in L^2(\Omega\mathbin{;}\mathbb R^3)$, let $q\in H^1_0(\Omega)$ be the solution of the problem
\begin{equation}
\label{hhnew0}
\int_\Omega\mu\nabla q\cdot\nabla v\,dx = \int_\Omega\mu\,\mathbf F\cdot \nabla v\,dx\,,\quad \text{for all } v\in H^1_0(\Omega)\,.
\end{equation}
Then, writing
\begin{subequations}\label{hh0}
\begin{equation}
\label{hh2}
\mathbf F=\nabla q+\zeta\,,
\end{equation}
we have  $\zeta\in X_\mu$ and
\begin{equation}
\label{hh3}
	\|\nabla q\|_{\Ld}\le\Lambda \|\mathbf F\|_{\Ld}\,,\quad \|\zeta\|_{\Ld}\le\Lambda \|\mathbf F\|_{\Ld}\,.
\end{equation}
\end{subequations}
Moreover, if $\mathbf F\in H^1(\Omega\mathbin{;}\mathbb R^3)$, with $\mathbf F\times n=0$ in $H^{-\frac{1}{2}}(\partial\Omega\mathbin{;}\mathbb R^3)$,
then $q\in H^2(\Omega)\cap H^1_0(\Omega)$ and we may take  $\zeta\in Y_\mu$.
\end{lm}

\begin{rmq}\label{const:mu}Clearly Lemma~\ref{lm:hh} is valid also if $\mu$ is replaced by any other function
for which property \eqref{1.2i} holds true; for example, it applies to constants. More precisely,
we can decompose any $\Ld$ vector field in the form $\mathbf F = \nabla q+\zeta$, where
$q\in H^1_0(\Omega)$ is the weak solution of $\Delta q  =\nabla\cdot \mathbf F$. In this case,
$\zeta$ has null (distributional) divergence, and if $\mathbf F$ belongs to $H^1(\Omega\mathbin{;}\mathbb R^3)$ then so does $\zeta$.
\end{rmq}

\subsection{Weak formulation}\label{ss:ws}
We fix a Lipschitz continuous function $\mu $ satisfying \eqref{1.2i}, we define the spaces $X_\mu$, $X$, $Y_\mu$, and $Y$,
as in \eqref{XmuYmu}, and we denote by $Y_\mu'$ the dual space of $Y_\mu$.
For $p\in[1,+\infty]$ and for every Hilbert space $Z$
we denote by $L^p(0,T\mathbin{;}Z)$ the space of all measurable functions $\mathbf F\colon[0,T]\to Z$ such that
\[
	\|\mathbf F \|_{L^p(0,T\mathbin{;}Z)}:=
	\begin{cases}
	\displaystyle\left(\int_0^T\|\mathbf F(t)\|_{Z}^p\,dt\right)^\frac{1}{p} & \quad \text{if }p<+\infty\,,\\ 
	\displaystyle\mathop{\rm ess\,sup}_{t\in[0,T]} \|\mathbf F(t)\|_{Z} & \quad \text{if }p=+\infty\,,
	\end{cases}
\]
is finite. We recall that $L^p(0,T\mathbin{;}Z)$ is a Banach space (uniformly convex if $p<+\infty$).
We shall need the following generalisation of a well known property of Sobolev space-valued mappings.
For a proof, one can repeat verbatim  the argument used in the proof of the analogous result in Sobolev spaces, see~\cite[Theorem 3, \S 5.9.2]{E}.

\begin{prop}\label{prop:Teo3E}
	Suppose that $\mathbf F\in L^2(0,T\mathbin{;} Y_\mu)$, with $\partial_t\mathbf F\in L^2(0,T\mathbin{;}Y_\mu')$. Then,
	by possibly redefining it on a negligible subset of $(0,T)$, the function $\mathbf F$ belongs to
	$C([0,T]\mathbin{;}X_\mu)$.
	Moreover, the mapping $t\mapsto \|\mathbf F(t)\|_{X_\mu}^2$ is absolutely continuous and for a.e.\ $t\in(0,T)$ we have
	\[
		\frac12\frac{d}{dt} \|\mathbf F(t)\|_{X_\mu}^2 = \langle \partial_t\mathbf F(t)\virg \mathbf F(t)\rangle_{Y_\mu'\times Y_\mu}\,.
\]
Eventually, there exists a constant $C$, depending only on $T$, such that
\begin{equation*}
	\sup_{t\in[0,T]} \|\mathbf F(t)\|_{X_\mu} \le C \Big( 
	\|\mathbf F\|_{L^2(0,T\mathbin{;}Y_\mu)} + \|\partial_t\mathbf F\|_{L^2(0,T\mathbin{;}Y_\mu') }
	\Big)\,.
\end{equation*}
If $\mathbf F\in L^2(0,T\mathbin{;}\Hu)$ and $\sigma\partial_t\mathbf F\in L^2(0,T\mathbin{;}\Hu')$
then 
\(
	\mathbf F\in C([0,T]\mathbin{;}L^2(\Omega\mathbin{;}\mathbb R^3))
	\),
for a.e.\ $t\in(0,T)$ we have
\[
		\frac12\frac{d}{dt} \int_\Omega\sigma \mathbf F(t)\cdot\mathbf F(t)\,dx = \langle \sigma\partial_t\mathbf F(t)\virg \mathbf F(t)\rangle\,,
\]
where $\langle\cdot\virg\cdot\rangle$ denotes now the pairing between $\Hu$ and its dual space $\Hu'$, and
\[
	\sup_{t\in[0,T]} \|\mathbf F(t)\|_{L^2} \le C \Big( 
	\|\mathbf F\|_{L^2(0,T\mathbin{;}\Hu)} + \|\partial_t\mathbf F\|_{L^2(0,T\mathbin{;}\Hu') }
	\Big)\,,
\]
where the constant $C$ depends on $\Lambda$ and $T$, only.
\end{prop}

\begin{dfn}\label{def-weak}
Given
\begin{equation}
\label{weak-source}
\Je\in L^2(0,T\mathbin{;}L^2(\Omega\mathbin{;}\mathbb R^3))\,, \ \Jm\in L^2(0,T\mathbin{;}X)\,,
\end{equation}
and
\begin{equation}
\label{weak-indata}
\mathbf H_0\in Y_\mu\,,
\end{equation}
we say that
$(\mathbf E\virg \mathbf H)\in L^2(0,T\mathbin{;} \Hu\times \Huz)$, with $\partial_t\mathbf H\in L^2(0,T;L^2(\Omega\mathbin{;}\mathbb R^3))$,
is a {\em weak solution of the eddy current system}
\begin{equation}
\label{1.10}
\begin{cases}
 \nabla\times \mathbf H -\sigma \mathbf E = \Je \,, &\quad \text{in $\Omega\times(0,T)$,}\\
 \nabla \times\mathbf E + \mu \partial_t \mathbf H = \Jm\,,& \quad \text{in $\Omega\times(0,T)$,}\\
 \mathbf H\times n = 0\,, & \quad \text{on $\partial\Omega\times(0,T)$,}\\
 \mathbf H=\mathbf H_0\,, & \quad \text{in $\Omega\times\{0\}$,}
\end{cases}
\end{equation} 
\begin{subequations}\label{weak}
 if for all $\varphi \in \Hu $ and for all $\psi\in \Huz$ we have
\renewcommand{\theequation}{\theparentequation \roman{equation}}
\begin{equation}\label{weak-ampere}
 \int_\Omega \mathbf H\cdot\nabla\times\varphi\,dx -\int_\Omega \sigma\mathbf E\cdot\varphi\,dx =
  \int_\Omega \Je\cdot \varphi\,dx
\end{equation}
\begin{equation}\label{weak-faraday}
\int_\Omega \mathbf E\cdot \nabla\times\psi\,dx + \int_\Omega \mu \partial_t\mathbf H\cdot \psi\,dx =
 \int_\Omega \Jm\cdot \psi\,dx
\end{equation}
for a.e. $t\in[0,T]$, and in addition we have
\begin{equation}
\label{weak0}
\mathbf H(0) = \mathbf H_0\,.
\end{equation}
\end{subequations}
\end{dfn}

\begin{rmq}\label{EvansC}
We note that \eqref{weak-source}, \eqref{weak-indata}, \eqref{weak-faraday}, and \eqref{weak0} imply that $\mathbf H\in L^2(0,T\mathbin{;}Y_\mu)$
and $\partial_t\mathbf H\in L^2(0,T\mathbin{;}X_\mu)$. Then,
$\partial_t\mathbf H\in L^2(0,T\mathbin{;}Y_\mu')$, due to the isometric embedding of $X_\mu$ into the dual $Y_\mu'$ of $Y_\mu$. Hence, in view of Proposition~\ref{prop:Teo3E},  we see that
$\mathbf H\in C([0,T]\mathbin{;}X_\mu)$ and thus equality \eqref{weak0} makes sense.
\end{rmq}

\begin{rmq}\label{rmq:smallertest}
Let equation \eqref{weak-faraday} hold for all $\psi\in Y_\mu$. Then, it holds for all $\psi\in \Huz $.
Indeed, by Lemma~\ref{lm:hh} we can write every $\psi\in C^1_0(\Omega\mathbin{;}\mathbb R^3)$
in the form $\psi = \nabla q+\zeta$ where $\zeta\in Y_\mu$ and 
\begin{equation}
\label{ampereconq}
\int_\Omega \mathbf E\cdot \nabla\times(\nabla q)\,dx =\int_\Omega \mu \partial_t\mathbf H\cdot \nabla q\,dx = \int_\Omega \Jm\cdot \nabla q\,dx=0\,,
\end{equation}
because $\nabla\times(\nabla q)=0$, and $\mu\partial_t\mathbf H,\Jm\in X$ for a.e.\ $t\in(0,T)$. Then, \eqref{weak-faraday} holds for
all test fields in $C^1_0(\Omega\mathbin{;}\mathbb R^3)$, which by~\cite[Remark 4.2]{DL} is dense in $\Huz$.
\end{rmq}

Formally, in view of the integration by parts formula \eqref{2.2}, 
a weak solution in the sense of Definition~\ref{def-weak} is
a solution to \eqref{1.1} with $\mathbf G=0$, satisfying the additional condition $\nabla\cdot(\mu\mathbf H)=0$.
Weak solutions in case of non-homogeneous boundary conditions are defined in the following sense.

\begin{dfn}\label{def-weak-G} Given $\Je,\Jm\in L^2(0,T\mathbin{;}L^2(\Omega\mathbin{;}\mathbb R^3))$, given $\mathbf H_0\in \Hu$, and
given
$\mathbf G\in L^2(0,T\mathbin{;} \Hu)$, 
with $\partial_t\mathbf G\in L^2(0,T\mathbin{;}L^2(\Omega\mathbin{;}R^3))$, 
such that for a.e.\ $t\in(0,T)$ we have
\begin{equation}
\label{compatibility}
\int_\Omega\left(\mu \mathbf G(x,t) - \mu\mathbf H_0(x)-\int_0^t\Jm(x,s)\,ds \right)\cdot\nabla u(x)\,dx=0\,,
\end{equation}
for all $u\in H^1_0(\Omega)$,
we say that
$(\mathbf E\virg \mathbf H)\in L^2(0,T\mathbin{;}\Hu^2)$, with $\partial_t\mathbf H\in L^2(0,T\mathbin{;}L^2(\Omega\mathbin{;}\mathbb R^3))$, is a  
{\em weak solution of the eddy current system}
\eqref{1.1}
if $\mathbf F:=\mathbf H-\mathbf G$ belongs to $ L^2(0,T\mathbin{;}\Huz)$ 
and $(\mathbf E\virg \mathbf F)$ solves, in the sense of Definition~\ref{def-weak}, the system
\begin{equation}
\label{eddyG}
	\begin{cases}
		\nabla\times\mathbf F -\sigma\mathbf E = \Je - \nabla\times \mathbf G & \qquad \text{in }\Omega\times(0,T)\,, \\
		\nabla\times\mathbf E+\mu\partial_t\mathbf F = \Jm- \mu\partial_t\mathbf G & \qquad \text{in }\Omega\times(0,T)\,,\\
		\mathbf F\times n=0 & \qquad \text{on } \partial\Omega\times(0,T)\,,\\
		\mathbf F=\mathbf H_0-\mathbf G & \qquad \text{in } \Omega\times\{0\}\,.
	\end{cases}
\end{equation}
\end{dfn}

We observe that Definition~\ref{def-weak-G} makes sense, because under
the assumptions made in Definition~\ref{def-weak-G} on $\Je$, $\Jm$, $\mathbf H_0$, and $\mathbf G$, it makes sense to consider weak solutions of \eqref{1.10}
in the sense of Definition~\ref{def-weak}, relative
to the sources
\[
	\widetilde{\Je}=\Je - \nabla\times \mathbf G\,, \qquad \widetilde{\Jm}=\Jm- \mu\partial_t\mathbf G\,,
\]
and to the initial datum
\[
\widetilde{\mathbf H}_0=	\mathbf H_0-\mathbf G(0)\,.
\]
Indeed, by \eqref{compatibility},
$\widetilde{\Je}$, $\widetilde{\Jm}$ satisfy
conditions \eqref{weak-source}.
Moreover,
since $\mathbf G\in L^2(0,T\mathbin{;}\Hu)$
and $\partial_t\mathbf G\in L^2(0,T;L^2(\Omega\mathbin{;}L^2(\Omega\mathbin{;}\mathbb R^3))$,
arguing as done in Remark~\ref{EvansC} we see that $\mathbf G$ belongs to
$C([0,T]\mathbin{;}L^2(\Omega\mathbin{;}\mathbb R^3))$, hence $\widetilde{\mathbf H}_0$ is  well-defined.
Eventually, again by \eqref{compatibility}, $\widetilde{\mathbf H}_0$ satisfies \eqref{weak-indata}.

\section{Existence and Uniqueness of Solutions}\label{s:ex}
The goal of the present section is to prove the following result.
\begin{thm}\label{teo1}
Let $\mathbf H_0\in Y_\mu$,
let $\Je\in L^2(0,T\mathbin{;}L^2(\Omega\mathbin{;}\mathbb R^3))$, with
$\partial_t\Je \in L^2(0,T\mathbin{;}L^2(\Omega\mathbin{;}\mathbb R^3))$, and let
$\Jm\in L^2(0,T\mathbin{;}X)$. Then, there exists a unique weak solution 
$(\mathbf E\virg\mathbf H)$ of 
\eqref{1.10}.
Moreover,
\begin{equation}
\label{teo1.1}
\begin{split}
	& \sup_{t\in[0,T]}   \|  \mathbf E(t)  \|^2_{\Ld}+ \sup_{t\in[0,T]}\|\mathbf H(t)\|_{\Ld}^2 + \int_0^T\|\partial_t\mathbf H(t)\|_{\Ld}^2\,dt\\
	& \qquad \le C \Big( \|\mathbf H_0\|_{\Hu}^2+\|\Je(0)\|_{\Ld}^2  +\int_0^T\big[ \|\Je(t)\|_{\Ld}^2+\|\Jm(t)\|_{\Ld}^2+\|\partial_t\Je(t)\|_{\Ld}^2\big]\,dt\Big)\,,
\end{split}
\end{equation}
where the constant $C$ depends on $\Lambda,T$, only.
\end{thm}

\begin{rmq}\label{veryweak}
When considering initial data $\mathbf H_0$ that belong merely to $X_\mu$, it is still possible to define solutions 
of \eqref{1.1}
in a weaker sense than that of Definition~\ref{def-weak},  
just requiring $\partial_t\mathbf H$ to take values in $Y_\mu'$ rather than in 
$X_\mu$, 
and replacing the
scalar product $(\partial_t\mathbf H\virg\psi)_{X_\mu}$ 
in the left hand-side of \eqref{weak-ampere} 
with the duality pairing $\langle\partial_t\mathbf H\virg \psi\rangle_{Y_\mu'\times Y_\mu}$. 
For a given $\mathbf H_0\in X_\mu\setminus Y_\mu$, the existence of solutions $(\mathbf E\virg\mathbf H)$ in this weaker sense
could be proved arguing similarly as done below to prove Theorem~\ref{teo1}, except that the final apriori estimate would be the following one
\begin{equation}
\label{energyestapriori}
\begin{split}
	\int_0^T \|  \mathbf E(t)  \|^2_{\Ld}\,dt  + \sup_{t\in[0,T]}\|\mathbf H(t)\|_{\Ld}^2& +\int_0^T\|\partial_t\mathbf H(t)\|_{Y_\mu'}^2\,dt\\
	& \le C \Big( \|\mathbf H_0\|_{\Ld}^2  +\int_0^T\big[ \|\Je(t)\|_{\Ld}^2+\|\Jm(t)\|_{\Ld}^2\big]\,dt\Big)\,,
\end{split}
\end{equation}
for a suitable constant $C$, again depending on $\Lambda$ and $T$, only.
\end{rmq}

\subsection{Magnetic eigenbase}\label{ss:2.1}
We fix $\mu\in W^{1,\infty}(\Omega)$ satisfying conditions \eqref{1.2i}.
\begin{lm}\label{lm:density}
The space $Y_\mu$ is dense in $X_\mu$, with respect to the weak convergence in $X_\mu$.
\end{lm}
\begin{proof}

We fix $\phi\in X_\mu$. By standard density results, there exists a sequence
$(\phi_i)\subset C^1_0(\Omega\mathbin{;}\mathbb R^3)$ with
\begin{equation}
\label{den1}
	\lim_{i\to\infty} \int_\Omega (\phi-\phi_i)\cdot \eta \,dx  = 0 \,, \quad \text{for all }\eta\in \Ld(\Omega\mathbin{;}\mathbb R^3)\,.
\end{equation}
By Lemma~\ref{lm:hh}, there exist $(q_i)\subset H^2(\Omega)\cap H^1_0(\Omega)$ and $(\zeta_i)\subset Y_\mu$ with
$\phi_i=\nabla q_i+\zeta_i$, and we have
\begin{equation}
\label{den1.5}
\int_\Omega\mu\, \nabla q_i\cdot\nabla v\,dx = \int_\Omega\mu\, \phi_i\cdot \nabla v\,dx\,,\qquad \text{for all }v\in H^1_0(\Omega)\,.
\end{equation}

 We prove that $(\zeta_i)$ converges to $\phi$ weakly in $L^2(\Omega\mathbin{;}\mathbb R^3)$.
To do so, by \eqref{den1}, it suffices to prove
\begin{equation}
\label{den2}
	\lim_{i\to\infty} \int_\Omega \mu\,\nabla q_i\cdot \eta \,dx  = 0 \,, 
\end{equation}
for all $\tilde\eta\in L^2(\Omega\mathbin{;}\mathbb R^3)$. We fix a test field $\tilde\eta$ and, using again Lemma~\ref{lm:hh}, we write $\tilde\eta = \nabla q+  \zeta$
for suitable $q\in H^1_0(\Omega)$ and $\zeta\in X_\mu$. Inserting $v = q$ in \eqref{den1.5} we obtain
\[
\int_\Omega \mu\,\nabla q_i\cdot \nabla q\,dx =  \int_\Omega\mu\, \phi_i\cdot \nabla q\,dx\,.
\]
Passing to the limit in the latter, using \eqref{den1}, and recalling that $\phi\in X_\mu$, we get
\begin{equation}
\label{den3}
	\lim_{i\to\infty} \int_\Omega\mu\, \nabla q_i\cdot \nabla q\,dx = \int_\Omega\mu\, \phi\cdot \nabla q\,dx =0\,.
\end{equation}
Since $q_i\in H^1_0(\Omega)$ for all $i\in\mathbb N$ and $\zeta\in X_\mu$, we also have
\begin{equation}
\label{den4}
\lim_{i\to\infty}\int_\Omega\mu\, \nabla q_i\cdot \zeta\,dx = 0\,.
\end{equation}
Summing \eqref{den3} and \eqref{den4} and recalling that  $\tilde\eta = \nabla q + \zeta$ we get \eqref{den2}. Since $\tilde\eta$ was arbitrary, we deduce
that $(\zeta_i)$ converges to $\phi$ weakly in $\Ld(\Omega\mathbin{;}\mathbb R^3)$. 
By \eqref{1.2i}, this implies that  $(\zeta_i)$ converges to $\phi$ with respect to the weak topology in $X_\mu$ relative to the scalar product \eqref{weight1}, too, as desired.
\end{proof}

The proof of the following spectral decomposition is based on standard methods, but we present it for sake of completeness.
\begin{lm}\label{Friedrichs}
There exists a sequence 
$0\le\lambda_1\le\lambda_2\le\ldots$, with $\lambda_i\to+\infty$ as $i\to\infty$, and a sequence $(\psi_i)\subset Y_\mu$, such that
$(\psi_i)$ is a complete orthonormal system in $X_\mu$ and for all $i\in\mathbb N$ we have
\begin{equation}
\label{eigenfri}
	\int_\Omega \mu\, \nabla\times\psi_i\cdot\nabla\times \phi\,dx 
	=\lambda_i \int_\Omega\mu\, \psi_i\cdot\phi\,dx\,,\quad \text{for all $\phi\in H^1(\Omega\mathbin{;}\mathbb R^3)$,}
\end{equation}
and $\psi_i\times n=0$ in $H^{-\frac{1}{2}}(\partial\Omega)$. Moreover, 
for every $i,j\in\mathbb N$ we have
\begin{equation}
\label{eigenfrion}
	\int_\Omega \mu\,\nabla \times\psi_i\cdot\nabla\times\psi_j\,dx = \lambda_{j}\delta_{ij}
\end{equation}
where $\delta_{ij}=1$ if $i=j$ and $\delta_{ij}=0$ otherwise.
\end{lm}

\begin{proof}
By Remark~\ref{rmq:norm} and Lax-Milgram Lemma, the linear operator $\mathcal{R}$ from $X_\mu$ to $X_\mu$ that takes every $\mathbf F\in X_\mu$ to
the corresponding solution $\psi\in Y_\mu$ of the following variational problem
\begin{equation}
\label{varprob}
	 \int_\Omega\mu\,\nabla\times\psi\cdot\nabla\times\phi\,dx+\int_\Omega\mu\,\psi\cdot\phi\,dx=\int_\Omega\mu\, \mathbf F\cdot \phi\,dx\,, 
	 \quad \text{for all } \phi\in Y_\mu\,,
\end{equation}
is well defined. Moreover, for every $\mathbf F\in X_\mu$, plugging in $\psi=\mathcal{R}\mathbf F$ in \eqref{varprob} yields 
\begin{equation}
\label{bddop}
\|\mathcal{R}\mathbf F\|_{Y_\mu}\le \|\mathbf F\|_{X_\mu}\,.
\end{equation}
Clearly $\|\mathcal{R}\mathbf F\|_{Y_\mu}\ge\|\mathcal{R}\mathbf F\|_{X_\mu}$. Then, by  \eqref{bddop},  $\mathcal{R}$  has 
 operator norm bounded by $1$.

We observe that $\mathcal{R}$ is injective. 
Indeed, by definition if $\mathbf F$ belongs to the kernel of $\mathcal{R}$ then $\psi=0$ is the solution of \eqref{varprob}. Thus, 
\(
	(\mathbf F\virg \phi)_{X_\mu}=0
\)
for all $\phi\in Y_\mu$. By Lemma~\ref{lm:density}, the latter holds in fact for all $\phi\in X_\mu$, hence $\mathbf F=0$.

Also, 
\(
(\mathbf F\virg\mathcal{R}\mathbf F)_{X_\mu}\ge0
\)
and
\(
(\mathbf F\virg \mathcal{R}\mathbf G)_{X_\mu}  = (\mathbf G\virg \mathcal{R}\mathbf F)_{X_\mu}
\), for every $\mathbf F,\mathbf G\in X_\mu$,
i.e., $\mathcal{R}$ is a positive and symmetric operator.

In addition, $\mathcal{R} $ is compact. Indeed,
given a bounded sequence $(\mathbf F_i)\subset X_\mu$, the sequence
$(\mathcal{R}\mathbf F_i)$ is bounded in $Y_\mu$  by \eqref{bddop}. By Remark~\ref{rmq:norm2},
it follows that 
$(\mathcal{R}\mathbf F_i)$ is precompact in $X_\mu$.

Therefore, $\mathcal{R} $ is a positive, compact, self-adjoint operator with trivial kernel from $X_\mu$ to itself,
having operator norm bounded by $1$. By the Spectral Theorem, there exists a sequence
$(\tau_i)\subset(0,1]$ and a Hilbert basis $(\psi_i)$ of $X_\mu$ with $\psi_i\in Y_\mu$ and $\mathcal{R}\psi_i=\tau_i\psi_i$ for all $i\in\mathbb N$, and the first statement
follows just setting $\lambda_i= \tau_i^{-1}-1$.

Eventually, we fix $i,j\in\mathbb N$, we test equation \eqref{eigenfri} with $\phi=\psi_j$, and we get
\[
	\int_\Omega\mu\, \nabla\times\psi_i\cdot\nabla\times\psi_j\,dx = \lambda_i\int_\Omega\mu\, \psi_i\cdot\psi_j\,dx\,.
\]
Since $(\psi_i)$ is orthonormal in $X_\mu$ with respect to \eqref{weight1}, this gives \eqref{eigenfrion} and concludes the proof.
\end{proof}

\begin{rmq}\label{prequo}
Incidentally, Lemma~\ref{Friedrichs}, implies in particular that the vector space
\begin{equation}
\label{aniso-tg0-harm}
H_\mu	= \Big\{h \in L^2(\Omega\mathbin{;}\mathbb R^3)\colon \nabla\cdot(\mu h)=0\,, \ \nabla\times h=0\,, \ h\times n=0\Big\}
\end{equation}
is finite-dimensional, because it consists of solutions of \eqref{eigenfri} corresponding to the null eigenvalue. In other words, the least eigenvalue
either equals zero or is positive depending on whether or not $\Omega$ supports non-trivial vector fields within \eqref{aniso-tg0-harm}.

We note that \eqref{aniso-tg0-harm} is trivial if $\Omega$ is {\em contractible}, i.e., if there exists $x_0\in\Omega$ and a function $g\in C^\infty([0,1]\times\Omega\mathbin{;}\Omega)$
with $g(0,\cdot)={\rm id}_\Omega$ and $g(1,x)=x_0$ for all $x\in\Omega$. For example, $\Omega$ has this property if it is simply connected and $\partial\Omega$ is connected; in this case, 
every $h\in H_\mu$ is the gradient of a scalar potential $w$, and $w$ is a weak solution
of the elliptic equation $\nabla\cdot (\mu\nabla w)  =0$ with homogeneous Dirichlet boundary conditions, hence it is a constant.
\end{rmq}

To prove Theorem~\ref{teo1}, we observe that
$\Hu$, with the scalar product induced by \eqref{2.1}, is a separable Hilbert space. Thus it admits a complete orthonormal system; we pick one, and we denote
it by
 $(\varphi_i)$. Then, let $(\psi_i)$ be the complete orthonormal system of $X_\mu$ introduced in Section~\ref{ss:2.1}, with $(\lambda_i)$ being the sequence
of all corresponding eigenvalues, counted with multiplicity.

\subsection{Approximate solutions}
Given $\Je\in L^2(0,T\mathbin{;}L^2(\Omega\mathbin{;}\mathbb R^3))$, $\Jm\in L^2(0,T\mathbin{;}X)$, and $\mathbf H_0\in X_\mu$, we set
\begin{equation}
\label{2.11}
	\mathbf H_{0m} = \sum_{j=1}^m (\mathbf H_0\virg \psi_j)_{X_\mu} \psi_j\,,
\end{equation}
and following Galerkin's scheme, we seek approximate solutions having the structure
\begin{equation}
\label{2.12}
	\mathbf E_m(t) = \sum_{j=1}^m \mathrm e_{jm}(t)\varphi_j\,,
\qquad	\mathbf H_m(t) = \sum_{j=1}^m \mathrm h_{jm}(t)\psi_j\,.
\end{equation}
More precisely, we prescribe the validity of the following $2m$ equations
\begin{subequations}\label{2.13}
\renewcommand{\theequation}{\theparentequation \roman{equation}}
\begin{equation}\label{2.13i}
\begin{split}
 \int_\Omega \nabla\times \mathbf H_m\cdot\varphi_i\,dx -\int_\Omega\sigma \mathbf E_m\cdot\varphi_i\,dx = \int_\Omega\Je\cdot \varphi_i\,dx\,,& \quad
 i=1,2,\ldots,m
\end{split}
\end{equation}
\begin{equation}\label{2.13ii}
\begin{split}
\int_\Omega \nabla\times\mathbf E_m\cdot \psi_i\,dx + \int_\Omega\mu\,  \partial_t\mathbf H_m\cdot \psi_i\,dx = \int_\Omega \Jm\cdot \psi_i\,dx\,, & \quad  i=1,2,\ldots,m
\end{split}
\end{equation}
\end{subequations}
and of the initial conditions
\begin{equation}\label{2.14}
	\mathbf H_m(0)=\mathbf H_{0m}\,.
\end{equation}

\begin{lm}\label{lm2.2} Let $\mathbf H_0\in X_\mu$. Then, there exists a unique solution 
\begin{equation}
\label{2.15}
	(\mathbf E_m\virg\mathbf H_m)\in C^1\left([0,T]\mathbin{;}{\rm Span}\{\varphi_1,\ldots,\varphi_m\}\times{\rm Span}\{\psi_1,\ldots,\psi_m\}\right)
\end{equation}
of the system \eqref{2.13} satisfying \eqref{2.14}. If in addition we have $\mathbf H_0\in Y_\mu$, then
\begin{equation}
\label{2.16}
\|\mathbf E_m(0)\|_{\Ld} \le C \big( \|\nabla \times \mathbf H_{0}\|_{\Ld} + \|\Je(0)\|_{\Ld}\big)\,.
\end{equation}
for a constant $C$ depending only on $\Lambda$.
\end{lm}

\begin{proof}
We write the system \eqref{2.13} in the form
\begin{equation}
\label{peso-ex-1}
	\begin{split}
		& (\nabla\times\mathbf H_m\virg \varphi_i)_{\Ld} - (\sigma\mathbf E_m\virg \varphi_i)_{\Ld} = (\Je\virg \varphi_i)_{\Ld}\,,\\
		& (\mu^{-1}\nabla	\times\mathbf E_m\virg \psi_i)_{X_\mu} + (\partial_t\mathbf H_m\virg \psi_i)_{X_\mu} =  (\mu^{-1}\Jm\virg\psi_i)_{X_\mu}\,.
	\end{split}
\end{equation}
Seeking solution with the structure \eqref{2.12} we are led to the $2m$ equations
\begin{subequations}
\label{peso-ex-2}
	\begin{align}
		\label{peso-ex-2i} 
			&\sum_{j=1}^m (\nabla\times\psi_j\virg \varphi_i)_{\Ld}{\mathrm h}_{jm}(t) - \sum_{j=1}^m(\sigma\varphi_j\virg\varphi_i)_{\Ld} \mathrm e_{jm}(t) = (\Je(t)\virg \varphi_i)_{\Ld}
			\quad  \text{for }i=1,\ldots,m\\
		\label{peso-ex-2ii}
			&  \sum_{j=1}^m(\nabla\times\varphi_j\virg \psi_i)_{\Ld}{\mathrm e}_{jm}(t)  +\sum_{j=1}^m (\mu\psi_j\virg \psi_i)_{\Ld} \frac{d}{dt}\mathrm h_{jm}(t) = (\Jm(t)\virg \psi_i)_{\Ld}
			 \quad \text{for }i=1,\ldots,m\,.
	\end{align}
\end{subequations}

By \eqref{1.2ii} and thanks to the fact that
$(\varphi_i)$ is a linearly independent system in $\Ld(\Omega\mathbin{;}\mathbb R^3)$, the quadratic form defined on $\mathbb R^m$ by
\begin{equation}
\label{peso-ex-fq}
	\mathcal{Q}(v)  = \sum_{i,j=1}^m(\sigma\varphi_j\virg\varphi_i)_{\Ld}v_iv_j\,, \qquad \text{for all }v\in\mathbb R^m\,,
\end{equation}
is positive definite and
$\mathcal{Q}(v)\ge \Lambda^{-1} |v|^2$, for all $v\in\mathbb R^m$.
The matrix $\{(\sigma\varphi_j\virg\varphi_i)_{\Ld}\}_{i,j=1}^m$ is symmetric because so is $\sigma$. Moreover, it is invertible and, denoting by $M^\sigma$ the inverse matrix (which is also symmetric), we have
\begin{equation}
\label{peso-ex-3}
	|M^\sigma v|\le \Lambda |v|^2\,,\qquad \text{for all }v\in\mathbb R^m\,.
\end{equation}
Then,
\eqref{peso-ex-2i} becomes
\begin{equation}
\label{peso-ex-4}
	\mathrm e_{im}(t) = \sum_{j,k=1}^mM^\sigma_{ik}(\nabla\times\psi_j\virg \varphi_k)_{\Ld} \mathrm h_{jm}(t) -\sum_{j=1}^m M^\sigma_{ij}(\Je(t)\virg\varphi_j)_{\Ld} \,,
	\quad
	i=1,\ldots,m\,.
\end{equation}

Since $(\psi_i)$ is an orthonormal system in $X_\mu$ with respect to the scalar product introduced in \eqref{weight1},  $(\mu\psi_i\virg \psi_j)_{\Ld}=\delta_{ij}$ for all
$i,j=1,\ldots,m$. Then \eqref{peso-ex-2ii} gives
\begin{equation}
\label{peso-ex-5}
	\frac{d}{dt}\mathrm h_{im} =-\sum_{j=1}^m (\nabla\times\varphi_j\virg \psi_i)_{\Ld} \mathrm e_{jm} + (\Jm\virg \psi_i)_{\Ld}\,, 
	\quad
	i=1,\ldots,m\,.
\end{equation}
Using \eqref{peso-ex-2i} to get rid of $\mathrm e_{jm}$ in \eqref{peso-ex-5}, we obtain
\begin{equation}
\label{peso-ex-5.5}
\begin{split}
	\frac{d}{dt}	\mathrm h_{im} 
	  = & -\sum_{j,k,\ell=1}^m   (\nabla\times\varphi_j\virg \psi_i)_{\Ld} 
	M^\sigma_{jk}(\nabla\times\psi_\ell\virg \varphi_k)_{\Ld} \mathrm h_{\ell m} \\
	& \qquad +\sum_{j,k=1}^m (\nabla\times\varphi_j\virg \psi_i)_{\Ld} M^\sigma_{jk}(\Je\virg\varphi_k)_{\Ld}
	+ (\Jm\virg \psi_i)_{\Ld}\,, \qquad i=1,\ldots,m\,.
\end{split}
\end{equation}

We set $\vec{\mathrm e}_m=(\mathrm e_{11}\virg\ldots\virg\mathrm e_{1m})$ and $\vec{\mathrm h}_m = (\mathrm h_{1m}\virg\ldots\virg\mathrm h_{mm})$.
We observe that, by \eqref{2.2}, for all $i,j=1,\ldots,m$ the scalar products
\(
(\nabla\times\psi_j\virg\varphi_i)_{\Ld}\) and \( (\nabla\times\varphi_i\virg\psi_j)_{\Ld}
\)
are equal and we denote by $A_{ij}$ their common value.
Then, the $m$ equations appearing in \eqref{peso-ex-5.5} can be recast in the form
\begin{equation}
\label{peso-ex-6}
	\frac{d}{dt}\vec {\mathrm h}_m = -A^TM^\sigma A \mathrm h_m + \vec{\mathrm b}_m\,,
\end{equation}
for a suitable $\vec{\mathrm b}_m\in L^2([0,T]\mathbin{;}\mathbb R^m)$.
By the standard existence theory for linear systems,
there exists $\vec {\mathrm h}_m \in C^1([0,T]\mathbin{;}\mathbb R^m)$ that solves  \eqref{peso-ex-6} for a.e.\ $t\in(0,T)$, with the initial conditions
\[
\vec{\mathrm h}_m(0)=((\mathbf H_0\virg\psi_1)_{X_\mu}\virg\ldots\virg(\mathbf H_0\virg \psi_m)_{X_\mu})\,.
\]
Then, we use \eqref{peso-ex-4} to define $\vec {\mathrm e}_m \in C^1([0,T]\mathbin{;}\mathbb R^m)$.
Therefore, by construction the functions $\mathbf E_m$ and $\mathbf H_m$ introduced in \eqref{2.12} are such that \eqref{peso-ex-1} is valid,
and the initial conditions \eqref{2.14} hold.

Now, we assume that $\mathbf H_0\in Y_\mu$. By \eqref{2.12} and \eqref{peso-ex-fq}, we have
\begin{equation}
\label{peso-ex-7}
(\sigma\mathbf E_m\virg \mathbf E_m)_{\Ld} = \mathcal{Q}(\vec{\mathrm e}_m)\,.
\end{equation}
Then we observe that  \eqref{peso-ex-4} implies
\begin{equation}
\label{peso-ex-8}
	\mathcal{Q}(\vec{\mathrm e}_m) = A \vec{\mathrm h}_m\cdot \vec{\mathrm e}_m - \sum_{i=1}^m (\Je\virg\varphi_i)_{\Ld} \mathrm e_{im}
	= (\nabla\times\mathbf H_m\virg \mathbf E_m)_{\Ld} -(\Je\virg \mathbf E_m)_{\Ld}\,,
\end{equation}
where in the second equality we simply used \eqref{2.12}. Since \eqref{peso-ex-7} and \eqref{peso-ex-8} holds, in particular, for $t=0$, we deduce that
\begin{equation}
\label{peso-ex-9}
	(\sigma\mathbf E_m(0)\virg\mathbf E_m(0))_{\Ld} 	= (\nabla\times\mathbf H_{0m}\virg \mathbf E_m(0))_{\Ld} - (\Je(0)\virg \mathbf E_m(0))_{\Ld}\,.
\end{equation}
By Cauchy-Schwartz inequality, we have
\[
	(\nabla\times\mathbf H_{0m}\virg \mathbf E_m(0))_{\Ld} - (\Je(0)\virg \mathbf E_m(0))_{\Ld}\le \Big[ \|\nabla\times\mathbf H_{0m}\|_{\Ld}+\|\Je(0)\|_{\Ld}\Big]
	\|\mathbf E_m(0)\|_{\Ld}\,.
\]
Using this and \eqref{1.2i}, from \eqref{peso-ex-9} we deduce
\begin{equation}
\label{peso-ex-10}
		\Lambda^{-1}\|\mathbf E_m(0)\|_{\Ld}\le  \|\nabla\times\mathbf H_{0m}\|_{\Ld}+\|\Je(0)\|_{\Ld}\,.
\end{equation}

By \eqref{1.2i}, we have
\begin{equation}
\label{peso-ex-11}
\|\nabla\times\mathbf H_{0m}\|_{\Ld}^2 
\le \Lambda (\nabla\times\mathbf H_{0m}\virg \nabla\times\mathbf H_{0m})_{X_\mu}\,.
\end{equation}
Thanks to \eqref{2.11}, \eqref{2.14}, and recalling \eqref{eigenfrion}, we obtain that
\begin{equation}
\label{peso-ex-12}
 (\nabla\times\mathbf H_{0m}\virg \nabla\times\mathbf H_{0m})_{X_\mu}
 = \sum_{i,j=1}^m (\mathbf H_0\virg \psi_i)_{X_\mu} (\mathbf H_0\virg \psi_j)_{X_\mu} (\nabla\times\psi_i\virg\nabla\times\psi_j)_{X_\mu}
=\sum_{i=1}^m \lambda_i|(\mathbf H_0\virg\psi_i)_{X_\mu}|^2\,.
\end{equation}
Since $\mathbf H_0\in Y_\mu$, by \eqref{eigenfri} we also have $\lambda_i(\mathbf H_0\virg\psi_i)_{X_\mu}=(\nabla\times\mathbf H_0\virg\nabla\times\psi_i)_{X_\mu}$. Hence
\begin{equation}
\label{peso-ex-13}
\sum_{i=1}^\infty \lambda_i|(\mathbf H_0\virg\psi_i)_{X_\mu}|^2 
=
\sum_{\lambda_i>0}\left|\big(\nabla\times\mathbf H_0\virg\lambda_i^{-\frac12}\nabla\times\psi_i\big)_{X_\mu}\right|^2 
\le\Lambda \|\nabla\times \mathbf H_0\|_{\Ld}^2\,,
\end{equation}
where in the last passage we also used Bessel's inequality and the fact that $(\lambda_i^{-1/2}\nabla\times\psi_i)$ is an orthonormal system in $X_\mu$, by \eqref{eigenfrion}. Clearly, \eqref{peso-ex-10}, \eqref{peso-ex-11}, \eqref{peso-ex-12},  and \eqref{peso-ex-13} imply \eqref{2.16} and this concludes the proof.
\end{proof}

\subsection{Energy estimates}
We provide ourselves with standard a priori bounds for the approximate solutions, so as to construct weak solutions by compactness.

\begin{prop}\label{prop:ee}
Let $\mathbf H_0\in X_\mu$ and let $(\mathbf E_m\virg\mathbf H_m)$ be as in Lemma~\ref{lm2.2}.
Then 
\begin{equation}
\label{teo1.2pre}
		\int_0^T \|  \mathbf E_m(t)\|_{\Ld}^2\, dt  + \sup_{t\in[0,T]} \|\mathbf H_m(t)\|_{\Ld}^2 
		\le C
		 \Big( \|\mathbf H_0\|_{\Ld}^2 +\int_0^T (\|\Je(t)\|_{\Ld}^2+\|\Jm(t)\|_{\Ld}^2)\,dt\Big)\,,
\end{equation}
for a constant $C>0$ depending on $\Lambda$, and $T$, only. If in addition $\mathbf H_0\in Y_\mu$ then
\begin{equation}
\label{teo1.2}
\begin{split}
	\sup_{t\in[0,T]} \| & \mathbf E_m(t)\|_{\Ld}^2\, + \sup_{t\in[0,T]} \|\mathbf H_m(t)\|_{\Ld}^2 + \int_0^T\|\partial_t \mathbf H_m(t)\|_{\Ld}^2\,dt\\
	& \le C \Big( \|\mathbf H_0\|_{\Hu}^2+ \|\Je(0)\|_{\Ld}^2 +\int_0^T (\|\Je(t)\|_{\Ld}^2+\|\Jm(t)\|_{\Ld}^2+\|\partial_t\Je(t)\|_{\Ld}^2)\,dt\Big)\,,
\end{split}
\end{equation}
for a (possibly different) constant $C>0$ depending on $\Lambda$, and $T$, only.
\end{prop}

\begin{proof}
By \eqref{2.13}, for all $(\varphi,\psi)\in{\rm Span}\{\varphi_1,\ldots,\varphi_m\} \mathbin{\times}{\rm Span}\{\psi_1,\ldots,\psi_m\}$ we have
\begin{subequations}\label{peso-en-0}
\renewcommand{\theequation}{\theparentequation \roman{equation}}
\begin{align} 
		(\nabla\times\mathbf H_m\virg\varphi)_{\Ld}- (\sigma\mathbf E_m\virg\varphi)_{\Ld} &= (\Je\virg\varphi)_{\Ld} \label{3.12i} \\
		(\nabla\times\mathbf E_m\virg \psi)_{\Ld} + ( \mu\partial_t \mathbf H_m\virg \psi)_{\Ld} &= (\Jm\virg\psi)_{\Ld}\label{3.12ii}\,.
\end{align}
\end{subequations} 
We divide now the proof into two steps.
\vskip.2cm

\step{1}{Core Energy inequality}
We observe that
\[
(\mu \partial_t\mathbf H_m\virg\mathbf H_m)_{\Ld} = \frac{1}{2}\frac{d}{dt} (\mu\mathbf H_m\virg\mathbf H_m)_{\Ld}\,.
\]
Then, choosing $\varphi = \mathbf E_m$  in \eqref{3.12i} and $\psi = \mathbf H_m$ in \eqref{3.12ii}, integrating on $(0,t)$, and using \eqref{2.14}, we obtain the energy identity
\begin{equation}
\label{peso-en-1}
	\frac{1}{2}\|\mathbf H_m(t)\|_{X_\mu}^2 + \int_0^t(\sigma\mathbf E_m\virg\mathbf E_m)_{\Ld} =\frac{1}{2} \|\mathbf H_{0m}\|_{X_\mu}^2+\int_0^t (\Jm\virg \mathbf H_m)_{\Ld}
	-\int_0^t(\Je\virg\mathbf E_m)\,.
\end{equation}
By Cauchy Schwartz and Young inequality we have
\[
(\Jm\virg\mathbf H_m)_{\Ld} \le \frac{1}{2}\|\mathbf H_m\|_{X_\mu}^2 + \frac{1}{2}\|\mu^{-1}\Jm\|_{X_\mu}^2\,.
\]
Using Cauchy-Schwartz inequality for the scalar product $(\phi\virg\psi)\mapsto(\sigma \phi\virg\psi)_{\Ld}$ induced by the symmetric matrix $\sigma$ and then
using Young's inequality again, we also have
\[
(\Je\virg\mathbf E_m)_{\Ld} \le \frac{1}{2}(\sigma\mathbf E_m\virg \mathbf E_m)_{\Ld} + \frac{1}{2}(\sigma^{-1}\Je\virg \Je)_{\Ld}\,.
\]

Also,  $\|\mathbf H_{0m}\|_{X_\mu}\le \|\mathbf H_0\|_{X_\mu}$ by \eqref{2.11}. Using these inequalities
in \eqref{peso-en-1}, together with \eqref{1.2i}, we get
\begin{equation}
\label{peso-en-2}
	\|\mathbf H_m(t)\|_{X_\mu}^2 + \int_0^t(\sigma\mathbf E_m\virg\mathbf E_m)_{\Ld}ds 
	\le \|\mathbf H_{0}\|_{X_\mu}^2+\Lambda\!\int_0^t \!\!\big( \|\Je\|_{\Ld}^2+\|\Jm\|_{\Ld}^2 \big)ds
	+\int_0^t \|\mathbf H_m(s)\|_{X_\mu}^2ds\,.
\end{equation}
By \eqref{2.15}, $t\mapsto\|\mathbf H_m(t)\|^2$ is continuous. Thus, by Gr\"onwall's Lemma, \eqref{peso-en-2} implies the inequality
\begin{equation}
\label{peso-en-3}
\|\mathbf H_m(t)\|_{X_\mu}^2 + \int_0^t(\sigma\mathbf E_m\virg\mathbf E_m)_{\Ld}ds 
	\le C\Big[ \|\mathbf H_{0}\|_{X_\mu}^2+\Lambda\!\int_0^t \!\!\big( \|\Je\|_{\Ld}^2+\|\Jm\|_{\Ld}^2 \big)ds\Big]\,,
\end{equation}
where $C$ is a constant depending on $T$, only. Using \eqref{structure}, from \eqref{peso-en-3} we deduce  that
\begin{equation}
\label{peso-en-4}
	\|\mathbf H_m(t)\|_{\Ld}^2 + \int_0^t\|\mathbf E_m\|_{\Ld}^2ds 
	\le C\Big[ \|\mathbf H_{0}\|_{\Ld}^2+\int_0^t \!\!\big( \|\Je\|_{\Ld}^2+\|\Jm\|_{\Ld}^2 \big)ds\Big]\,,\quad \text{for all }t\in[0,T]\,.
\end{equation}
for an appropriate constant $C$, depending only on $\Lambda$ and $T$. This implies \eqref{teo1.2pre}.
\vskip.2cm

\step{2}{Estimate of $\partial_t\mathbf H_m$}
Differentiating in \eqref{3.12i} with respect to $t$ and taking $\varphi = \mathbf E_m$ in the resulting equation, we get
\begin{equation}
\label{peso-en-5}
	(\nabla\times\partial_t\mathbf H_m\virg\mathbf E_m)_{\Ld}-(\sigma\mathbf E_m\virg\partial_t \mathbf E_m)_{\Ld} = (\partial_t\Je\virg\mathbf E_m)_{\Ld}\,.
\end{equation}
Choosing $\psi = \partial_t\mathbf H_m$ in \eqref{3.12ii}, we obtain
\begin{equation}
\label{peso-en-6}
	(\nabla\times\mathbf E_m\virg\partial_t\mathbf H_m)_{\Ld} + (\mu\partial_t\mathbf H_m\virg\partial_t\mathbf H_m)_{\Ld} = (\Jm\virg\partial_t\mathbf H_m)_{\Ld}\,.
\end{equation}
Moreover,  $\partial_t\mathbf H_m $ takes values in $\Huz$. Hence, $	(\nabla\times\mathbf E_m\virg\partial_t\mathbf H_m)_{\Ld}=
	(\mathbf E_m\virg\nabla\times\partial_t\mathbf H_m)_{\Ld}$. Then, subtracting \eqref{peso-en-5} from \eqref{peso-en-6} 
	and integrating over $(0,t)$ we obtain
\begin{equation*}
	\int_0^t\!\!\|\partial_t\mathbf H_m\|_{X_\mu}^2+ \tfrac{1}{2}(\sigma\mathbf E_m(t)\virg\mathbf E_m(t))_{\Ld}
	=\tfrac{1}{2}(\sigma\mathbf E_m(0)\virg\mathbf E_m(0))_{\Ld}+\int_0^t\!\![ (\Jm\virg\partial_t\mathbf H_m)_{\Ld}-(\partial_t\Je\virg\mathbf E_m)_{\Ld}]\,.
\end{equation*}
By Cauchy-Schwartz and Young inequality,
\[
(\Jm\virg\partial_t\mathbf H_m)_{\Ld}\le \tfrac{1}{2}\|\partial_t\mathbf H_m\|_{X_\mu}^2 + \tfrac{1}{2}(\mu^{-1}\Jm\virg\Jm)_{\Ld}\,,\ \text{and }\
(\partial_t\Je\virg\mathbf E_m)_{\Ld}\le \tfrac{1}{2}\|\partial_t\Je\|_{\Ld}^2+\tfrac{1}{2}\|\mathbf E_m\|^2_{\Ld}\,.
\]
By these inequalities and \eqref{1.2ii}, the previous identity implies that for a.e.\ $t\in(0,T)$ the inequality
\[
\|\mathbf E_m(t)\|_{\Ld}^2+ \int_0^t\|\partial_t\mathbf H_m\|_{\Ld}^2 \le
C\Big[  \|\mathbf E_m(0)\|_{\Ld}^2 +
  \int_0^T\|\mathbf E_m\|^2+
   \int_0^T\big[\|\partial_t\Je\|_{\Ld}^2+\|\Jm\|_{\Ld}^2\big]\,\Big]\,,
\]
holds, with a constant $C$ depending only on $\Lambda$. Eventually, recalling \eqref{2.16}, from the last inequality and \eqref{peso-en-4} we
deduce \eqref{teo1.2}, as desired.
\end{proof}

\subsection{Proof of Theorem~\ref{teo1}}
We first assume that $\Je=\Jm=0$ for a.e.\ $t\in(0,T)$, and that $\mathbf H_0=0$. Then we test equation \eqref{weak-ampere} with $\varphi=\mathbf E$
and \eqref{weak-faraday} with $\psi=\mathbf H$. By \eqref{2.2} and by an integration in time we arrive at
\[
	(\mu\mathbf H(t)\virg\mathbf H(t))_{\Ld} + \int_0^t(\sigma\mathbf E(s)\virg\mathbf E(s))_{\Ld}\,ds=0\,, \qquad \text{for all }t\in[0,T].
\]
By \eqref{structure}, both the first summand and the integrand in the second one are positive quantities.
Then, $\mathbf H(t)=\mathbf E(t)=0$ for a.e.\ $t\in[0,T]$. By linearity this implies at once the uniqueness statement.

Now we prove the existence of solutions. for every $m\in\mathbb N$, let $(\mathbf E_m\virg \mathbf H_m)$ be as in Lemma~\ref{lm2.2}.
The energy estimate \eqref{teo1.2} of Proposition~\ref{prop:ee} implies that,
by possibly passing to a subsequence,
\begin{equation}
\label{teo1.3}
	\begin{split}
			\mathbf E_m\rightharpoonup \mathbf E & 
			\quad \text{weakly-$\ast$ in $L^\infty(0,T;X_\mu)$,}\\
		\mathbf H_m \rightharpoonup \mathbf H & \quad \text{weakly-$\ast$ in $L^\infty(0,T;X_\mu)$,}\\
		\mathbf \partial _t\mathbf H_m\rightharpoonup \partial_t\mathbf H & \quad \text{weakly in $L^2(0,T;X_\mu)$.}
	\end{split}
\end{equation}
Clearly \eqref{teo1.2} and \eqref{teo1.3} imply the estimate \eqref{teo1.1}. We are left to prove that the limit $(\mathbf E\virg\mathbf H)$ is a
weak solution of \eqref{1.10}.

For all functions $\varphi\in \Hu$ and $\psi\in Y_\mu$ that take the form
\begin{equation}
\label{teo1.4}
	\varphi(x,t) = \sum_{i=1}^N \alpha_i(t) \varphi_i(x)\,,\qquad \psi(x,t) = \sum_{i=1}^N \beta_i(t)\psi_i(x)
\end{equation}
for some $\alpha_i,\beta_i\in C^\infty([0,T])$ and $N\in\mathbb N$, by \eqref{2.2} and \eqref{2.13} for all $m\ge N$ we have
\begin{subequations}\label{teo1.35}
\begin{align}
		\int_0^T\int_\Omega
		\mathbf H_m\cdot\nabla\times\varphi\,dx \,dt- 
		\int_0^T\int_\Omega
		\sigma\mathbf E_m\cdot\varphi\,dx \,dt=
		 \int_0^T\int_\Omega
		 \Je\cdot\varphi\,dxdt \\
\label{weak-finite-faraday}	
\int_0^T\int_\Omega
\mathbf E_m\cdot\nabla\times \psi \,dx\, dt+ 
\int_0^T\int_\Omega
\mu \partial_t \mathbf H_m\cdot\psi \,dx\,dt  =
 \int_0^T \int_\Omega
\Jm\cdot\psi\,dx\,dt\,.
\end{align}
\end{subequations}
Owing to \eqref{teo1.3}, from \eqref{teo1.35} we infer that
\begin{equation}
\label{teo1.5}
		\begin{split}
		\int_0^T\int_\Omega\mathbf H\cdot\nabla\times\varphi\,dx \,dt- \int_0^T\int_\Omega\sigma\mathbf E\cdot\varphi\, dx \,dt=
		 \int_0^T\int_\Omega\Je\cdot\varphi\,dx\,dt\,, \\
	\int_0^T\int_\Omega\mathbf E\cdot\nabla\times \psi \,dx\, dt+ \int_0^T\int_\Omega\mu\partial_t \mathbf H\cdot\psi \,dx\,dt  =\int_0^T \int_\Omega\Jm\cdot\psi\,dx\,dt\,.
		\end{split}
\end{equation}

The pairs $(\varphi,\psi)$ of the form \eqref{teo1.4} form
a dense set in $L^2(0,T\mathbin{;}\Hu\times Y_\mu)$. Thus,  
from \eqref{teo1.5} we deduce that, for a.e.\ $t\in[0,T]$,
\eqref{weak-ampere} holds for all $\varphi\in \Hu$  and \eqref{weak-faraday} holds for all $\psi\in Y_\mu$.
In view of Remark~\ref{rmq:smallertest}, it follows that \eqref{weak-faraday} holds for all $\psi\in \Huz$.

For a.e.\ $t\in (0,T)$, \eqref{weak} holds for all $\varphi\in \Hu$ and for all $\psi\in\Huz$
and this implies that $(\mathbf E\virg\mathbf H)\in L^2(0,T\mathbin{;}\Hu\times\Huz)$. By \eqref{teo1.3}
we also have $\partial_t\mathbf H\in L^2(0,T\mathbin{;} \Ld(\Omega\mathbin{;}\mathbb R^3))$.

Then, according to Definition~\ref{def-weak} (see also Remark~\ref{EvansC}) we are left to prove that \eqref{weak0} holds. To do so,
we fix $\psi \in C^1([0,T]\mathbin{;}\Huz)$, with $\psi(T)=0$. By \eqref{weak-faraday}, we have
\begin{equation}
\label{teo1.9}
	\int_0^T\int_\Omega \mathbf E\cdot\nabla\times\psi\,dx\,dt  -\int_0^T \int_\Omega \mu\,\mathbf H\cdot\partial_t\psi\,dx\,dt = \int_0^T\int_\Omega\Jm\cdot\psi\,dx\,dt
	+\int_\Omega \mu\,\mathbf H(0)\cdot \psi(0)\,dx\,.
\end{equation}
Also, by \eqref{weak-finite-faraday} we have 
\begin{equation}
\label{teo1.10}
	\int_0^T\int_\Omega \mathbf E_m\cdot\nabla\times\psi\,dx\,dt  -\int_0^T \int_\Omega \mu\,\mathbf H_m\cdot\partial_t\psi\,dx\,dt = \int_0^T\int_\Omega\Jm_m\cdot\psi\,dx\,dt
	+\int_\Omega \mu\,\mathbf H_{0m}\cdot \psi(0)\,dx\,.
\end{equation}
By \eqref{teo1.3}, passing to weak limits in \eqref{teo1.10} and comparing with \eqref{teo1.9}
we get that 
\[
	(\mathbf H(0)\virg \psi(0))_{X_\mu} = (\mathbf H_0\virg\psi(0))_{X_\mu}\,.
\]
Since $\psi(0)$ can be any element of $Y_\mu$,  by Lemma~\ref{lm:density} we deduce \eqref{weak0} and this ends the proof.
\qed

\section{Global H\"older estimates for the Magnetic Field}\label{s:reg}

Given $\alpha\in(0,1]$, by $C^{0,\alpha}(\overline\Omega)$
we denote the space of all continuous functions $u$ that are $\alpha$-H\"older continuous on $\overline\Omega$, meaning that

\centerline{
\(
\displaystyle
	\|u\|_{C^{0,\alpha}(\overline\Omega)} := \|u\|_{L^\infty(\Omega)} + \sup_{\substack{x,y\in \overline\Omega\\ x\neq y}} \frac{|u(x)-u(y)|}{|x-y|^\alpha}<+\infty\,.
\)
}
We recall that $C^{0,\alpha}(\overline\Omega) $ is a Banach space with this norm. The previous definition extends obviously to the case of vector-valued, and tensor-valued functions.
\begin{thm}\label{teoreg2}
There exists $\alpha_0\in(0,\frac{1}{2}]$, only depending on $\Lambda$, such that
for every $\alpha\in(0,\alpha_0]$ the following holds:
for every $\mathbf H_0\in C^{0,\alpha}(\overline\Omega\mathbin{;}\mathbb R^3)$ and for every $\Je,\Jm\in L^2(0,T\mathbin{;}C^{0,\alpha}(\overline\Omega\mathbin{;}\mathbb R^3))$,
if $(\mathbf E\virg\mathbf H)$ is a weak solution of
\eqref{1.10}, then $\mathbf H\in L^2(0,T\mathbin{;}C^{0,\alpha}(\overline\Omega\mathbin{;}\mathbb R^3))$, and we have
\begin{equation}
\label{holdestH}
\begin{split}
\| \mathbf H  (t) \|_{C^{0,\alpha}(\overline\Omega\mathbin{;}\mathbb R^3)}  \le C\Big[ & 
\|\mu\mathbf H_0\|_{C^{0,\alpha}(\overline\Omega\mathbin{;}\mathbb R^3)}+
\|\mathbf E(t)\|_{\Ld}+\|\mathbf H\|_{\Ld}+\|\mu\partial_t\mathbf H(t)\|_{\Ld} \\
 & \qquad +\int_0^t\|\Jm(s)\|_{C^{0,\alpha}(\overline\Omega\mathbin{;}\mathbb R^3)}ds +\|\Jm(t)\|_{\Ld}+\|\Je(t)\|_{C^{0,\alpha}(\overline\Omega\mathbin{;}\mathbb R^3)}
\Big]\,,
\end{split}
\end{equation}
for a.e. $t\in(0,T)$, where the constant $C$ depends on $\Lambda$ and on $r$.
\end{thm}

\subsection{Tools: Morrey and Campanato spaces}
For every $\lambda>0$, given $u\in L^2(\Omega)$ we say that $u$ belongs to Morrey's space $ L^{2,\lambda}(\Omega)$ if 
\[
	[u]_{L^{2,\lambda}(\Omega)}^2 := \sup_{\substack{x_0\in \Omega\\ \rho>0}} \rho^{-\lambda} \int_{B_\rho(x_0)\cap\Omega} |u|^2\,dx<+\infty\,.
\]
In this case we also write $\|u\|_{L^{2,\lambda}(\Omega)} = \|u\|_{L^2(\Omega;\mathbb R^3)}+[u]_{L^{2,\lambda}(\Omega)}$. We say that
$ u\in \mathcal{L}^{2,\lambda}(\Omega)$ if 
\[
	[u]_{\mathcal L^{2,\lambda}(\Omega)}^2 := \sup_{\substack{x_0\in \Omega\\ \rho>0}} \rho^{-\lambda} \int_{B_\rho(x_0)\cap\Omega} \left|u(x)-\frac{1}{|B_\rho(x_0)\cap\Omega|}\int_{B_\rho(x_0)\cap\Omega} u(y)\,dy\right|^2\,dx<+\infty\,,
\]
and in this case $\|u\|_{\mathcal L^{2,\lambda}(\Omega)} = \|u\|_{L^2(\Omega;\mathbb R^3)}+[u]_{\mathcal L^{2,\lambda}(\Omega)}$. For vector- and tensor-valued functions, Morrey's and Campanato's spaces
are defined similarly.

The space $\mathcal{L}^{2,\lambda}(\Omega)$ was introduced by Campanato in~\cite{C}. 
If for all $x_0\in \partial\Omega$ and for all $\rho>0$ we have\footnote{For example, this measure density requirement is met by all open set satisfying an interior cone condition.
In particular, clearly, it follows from assumption \eqref{hpOmega}.
} $|\Omega\cap B_\rho(x_0)|\ge K \rho^3$, with a constant $K$ depending only on $\Omega$, then
Campanato's space
is isomorphic to $L^{2,\lambda}(\Omega)$ for every $\lambda\in(0,3)$, to $C^{0,\frac{\lambda-3}{2}}(\overline\Omega)$ for every $\lambda\in(3,5]$. It can be seen that
it only consists of constant functions for every $\lambda>5$ and that it coincides with the space of BMO functions if $\lambda=3$, but this will be of no use in the sequel.

\subsection{Energy estimates}
In this section we provide some elementary a priori estimate for the eddy current sytstem.

\begin{lm}\label{teoenerg}
Let $\mathbf H_0\in X_\mu$, and let
 $(\mathbf E\virg\mathbf H)$ be a weak solution of \eqref{1.10} in the sense of Remark~\ref{veryweak}. Then
estimate \eqref{energyestapriori} holds  with a constant $C$ depending on $\mu$, $\Lambda$, and $T$, only.
\end{lm}
\begin{proof} Let $t\in(0,T)$ be such that
\eqref{weak} holds for all $(\varphi\virg\psi)\in\Hu\times\Huz$. Inserting $\varphi = \mathbf E$ in \eqref{weak-ampere} and $\psi = \mathbf H$ in \eqref{weak-faraday} and using \eqref{2.2} we obtain
\[
\int_\Omega \mu\partial_t\mathbf H\cdot\mathbf H\,dx +\int_\Omega \sigma \mathbf E\cdot\mathbf E\,dx= \int_\Omega \Jm\cdot \mathbf H\,dx
-\int_\Omega \Je\cdot \mathbf E\,dx\,.
\]
Using \eqref{1.2ii} to estimate from below the left hand-side, and Young inequality to estimate from above the right hand-side,
we obtain, for all given $\delta\in(0,1)$, that
\[
	\frac{1}{2}\frac{d}{dt}\int_\Omega \mu |\mathbf H|^2\,dx + \frac{1}{\Lambda}\int_\Omega |\mathbf E|^2\,dx \le 
	\frac{1}{2}\int_\Omega |\mathbf \Jm|^2\,dx +\frac{1}{2}\int_\Omega |\mathbf H|^2\,dx
	+\frac{\delta}{\Lambda}\int_\Omega |\mathbf E|^2\,dx + \frac{\Lambda}{4\delta} \int_\Omega |\Je|^2\,dx\,.
\]
Choosing $\delta = 1/2$ we absorb a term in the left hand-side. Then an integration gives
\begin{equation}
\label{pregro}
\begin{split}
 \int_\Omega & |\mathbf H(t)|^2\,dx  -\int_\Omega |\mathbf H_0|^2\,dx + \int_0^t\int_\Omega|\mathbf E|^2\,dx\,ds
 \\
 & 
 \le \Lambda^2 \Big[
 	\int_0^t\int_\Omega |\mathbf H|^2\,dx\,ds +  	\int_0^t\int_\Omega |\Jm|^2\,dx\,ds + 	\int_0^t\int_\Omega |\Je|^2\,dx\,ds 
 \Big]\,.
\end{split}
\end{equation}
By definition of weak solution (see Definition~\ref{def-weak} and Remark~\ref{EvansC}),
 $\mathbf H\in L^2(0,T\mathbin{;}Y_\mu) $ and $\partial_t\mathbf H\in L^2(0,T\mathbin{;}Y_\mu')$. In view of Proposition~\ref{prop:Teo3E},we have $\mathbf H\in C([0,T]\mathbin{;}\Ld)$, and the function
\[
	t\longmapsto \int_\Omega |\mathbf H(t) |^2\,dx\,,
\]
appearing in \eqref{pregro}, is absolutely continuous. Then, applying Gr\"onwall's Lemma, we obtain that
\[
\begin{split}
 \int_\Omega|\mathbf H(t)|^2\,dx & -\int_\Omega |\mathbf H_0|^2\,dx + \int_0^t\int_\Omega|\mathbf E|^2\,dx\,ds \le C \Big[	\int_0^T\int_\Omega |\Jm|^2\,dx\,ds + 	\int_0^T\int_\Omega |\Je|^2\,dx\,ds 
 \Big]
\end{split}
\]
for a suitable constant $C>0$, depending on $\mu$, $\Lambda$, and $T$, only. Since this procedure can be repeated for a.e.\ $t\in(0,T)$, we deduce \eqref{energyestapriori}.
\end{proof}
\begin{thm}\label{regapriori2}
Let $\mathbf H_0\in Y_\mu$, 
let $\Je\in L^2(0,T\mathbin{;}L^2(\Omega\mathbin{;}\mathbb R^3))$, with
$\partial_t\Je \in L^2(0,T\mathbin{;}L^2(\Omega\mathbin{;}\mathbb R^3))$, let
$\Jm\in L^2(0,T\mathbin{;}X)$, and let
 $(\mathbf E\virg\mathbf H)$ be a weak solution of 
\eqref{1.10} in the sense of Definition~\ref{def-weak}. Then
\[
\begin{split}
\sup_{t\in[0,T]} \|\mathbf E(t)\|_{\Ld}^2+\sup_{t\in[0,T]} \|\mathbf H(t)\|_{\Ld}^2  & +\int_0^T \|\partial_t\mathbf H(t)\|_{\Ld}^2\,dt
\le C \Big[\|\mathbf H_0\|_{\Ld}^2
  \\
 & \qquad+
\int_0^T \big(
	\|\Je(t)\|_{\Ld}^2+\|\Jm(t)\|_{\Ld}^2+\|\partial_t\Je(t)\|_{\Ld}^2
\big)\,dt
\Big]
\end{split}
\]
where the constant
$C$ depends on $\Lambda$ and $T$, only.
\end{thm}

\begin{proof}
Let $\varphi\in \Hu$. Differentiating with respect to $t$ in \eqref{weak-ampere} we obtain
\begin{equation}
\label{d-amp-d-t}
\int_\Omega \partial_t\mathbf H\cdot\nabla\times \varphi\,dx-	\langle \sigma\partial_t\mathbf E\virg \varphi\rangle =\int_\Omega \partial_t\Je\cdot\varphi\,dx\,,
\end{equation}
where $\langle\cdot\virg\cdot\rangle$ stands for the pairing between $\Hu$ and its dual space. Since, in \eqref{d-amp-d-t}, $\varphi$ is arbitrary, by \eqref{1.2ii} and by a density argument we deduce that 
that
\begin{equation}
\label{Etlim}
\int_0^T\!\!	\langle\sigma\partial_t\mathbf E\virg\,v\rangle
\le 
 \Big[ \int_0^T \|\partial_t\mathbf H\|_{(\Hu)'}^2+ \int_0^T \|\partial_t\Je\|_{(\Hu)'}^2\Big]^\frac{1}{2}
\|v\|_{L^2(0,T\mathbin{;}\Hu)}\,,
\end{equation}
for all $v\in L^2(0,T\mathbin{;}\Hu)$.
Then, as a function taking values in the dual space of $\Hu$,  $\partial_t\mathbf E$
is $\Ld$ on the interval $(0,T)$.
In view of Proposition~\ref{prop:Teo3E}, this gives
$\mathbf E\in C([0,T]\mathbin{;}L^2(\Omega;\mathbf R^3))$ and 
\begin{equation}
\label{anche-d-t-E}
	\frac{d}{dt} \int_\Omega\sigma\mathbf E(t)\cdot\mathbf E(t)\,dt= 2\langle\sigma \partial_t\mathbf E\virg\mathbf E\rangle\,,\quad \text{for a.e.}\ t\in(0,T),
\end{equation}
where $\langle\cdot\virg\cdot\rangle$ denotes the duality pairing between $\Hu$ and its dual space. Now we take $\varphi  = \mathbf E$ in \eqref{d-amp-d-t}, which we can do for a.e.\ $t\in(0,T)$. As a result, by \eqref{anche-d-t-E} we get
\[
\int_\Omega \partial_t \mathbf H \cdot \nabla\times\mathbf E\,dx  - \frac{1}{2}\frac{d}{dt}\int_\Omega \sigma\mathbf E\cdot\mathbf E\,dx = \int_\Omega \partial_t\Je\cdot\mathbf E\,dx\,.
\]
Also, for a.e.\ $t\in(0,T)$ we can test \eqref{weak-faraday} with $\psi  = \partial_t\mathbf H$, and doing so we get
\[
	\int_\Omega \mathbf E\cdot\nabla\times\partial_t\mathbf H\,dx + \int_\Omega \mu\,\partial_t\mathbf H\cdot\partial_t\mathbf H\,dx = \int_\Omega \Jm\cdot\partial_t\mathbf H\,dx\,.
\]
We observe that \eqref{2.2} implies
\[
	\int_\Omega \mathbf E\cdot\nabla\times\partial_t\mathbf H \,dx = \int_\Omega \partial_t \mathbf H \cdot \nabla\times\mathbf E\,dx  \,.
\]
Combining the last three identities we get
\[
\int_\Omega \mu\,\partial_t\mathbf H\cdot\partial_t\mathbf H\,dx +  \frac{1}{2}\frac{d}{dt}\int_\Omega \sigma\mathbf E\cdot\mathbf E\,dx
= -  \int_\Omega \partial_t\Je\cdot\mathbf E\,dx+\int_\Omega \Jm\cdot\partial_t\mathbf H\,dx\,.
\]
Integrating this energy identity over the interval $[0,t]$, using \eqref{structure} and Young's inequality we obtain
\[
\int_0^t\|\partial_t\mathbf H\|_{\Ld}^2 +  \|\mathbf E(t)\|^2_{\Ld}\le C\Big[ \|\mathbf E(0)\|^2_{\Ld}+\int_0^t\|\mathbf E\|^2+ \int_0^t\big(\|\partial_t\Je\|^2+\|\Jm\|^2\big)\Big]
\]
for a suitable C depending only on $\Lambda$. By Gr\"onwall's Lemma, we deduce that
\begin{equation}
\label{conE0}
\int_0^t\|\partial_t\mathbf H\|_{\Ld}^2 +  \|\mathbf E(t)\|^2_{\Ld}\le C\Big[ \|\mathbf E(0)\|^2_{\Ld}+ \int_0^t\big(\|\partial_t\Je\|^2+\|\Jm\|^2\big)\Big]\,,
\end{equation}
where the constant depends now on $\Lambda$ and $T$, only.

In order to get rid of the term depending on $\mathbf E(0)$ in the right hand-side of \eqref{conE0}, we note that by Proposition~\ref{prop:Teo3E} we also have
\[
\sup_{t\in[0,T]}\|\mathbf E(t) \|_{\Ld}^2 \le C \Big[ \int_0^T\|\mathbf E\|_{\Hu}^2 + \int_0^T \|\partial_t\mathbf E\|_{(\Hu)'}^2\Big]\,,
\]
with a constant depending only on $\Lambda$, and $T$. We also recall that by
\eqref{weak-faraday} we have
\[
\|\mathbf E\|_{\Hu}^2  = \| \mathbf E\|_{\Ld}^2
+\|\nabla\times\mathbf E\|_{\Ld}^2 \le 
	\|\mathbf E\|_{\Ld}^2+\|\mu\partial_t\mathbf H\|_{\Ld}^2+\|\Jm\|^2_{\Ld}\,,
\]
whereas \eqref{Etlim} implies
\[
\int_0^T \|\partial_t\mathbf E\|_{(\Hu)'}^2\le \int_0^T \|\partial_t\mathbf H\|^2+\|\partial_t\Je\|^2\,.
\]
Then, by Gr\"onwall Lemma it follows that
\begin{equation}
\label{supE}
	\sup_{t\in[0,T]} \|\mathbf E(t)\|_{\Ld}^2 \le C\Big[ 
		\int_0^T\|\partial_t\mathbf H\|_{\Ld}^2 + \|\partial_t\Je\|_{\Ld}^2+\|\Jm\|_{\Ld}^2
	\Big]\,,
\end{equation}
where $C$ depends on $\Lambda$ and $T$, only.

Inserting \eqref{supE} in \eqref{conE0} we arrive at
\[
\int_0^t\|\partial_t\mathbf H\|_{\Ld}^2 +  \|\mathbf E(t)\|^2_{\Ld}\le C
\int_0^T\big(\|\partial_t\mathbf H\|_{\Ld}^2 + \|\partial_t\Je\|_{\Ld}^2+\|\Jm\|_{\Ld}^2\big)\,,
\]
for a.e.\ $0\le t\le T$.
\end{proof}

\subsection{Proof of Theorem~\ref{teoreg2}} 
We set $$\mathcal{I}=\{t\in[0,T]\colon \mathbf E(t)\in \Hu\,,\ \mathbf H(t)\in Y_\mu\,, \ \Je(t)\,, \Jm(t)\in C^{0,\alpha}(\overline\Omega\mathbin{;}\mathbb R^3)\} $$ and we recall that $[0,T]\setminus \mathcal{I} $ is a negligible set (see Remark~\ref{EvansC}).
We drop the dependance on $t$ of the vector fields, so as to abbreviate the notations.

By Lemma~\ref{HelmE},  there exist $u\in H^1(\Omega) $ and $\eta \in H^1(\Omega\mathbin{;}\mathbb R^3)$ with
\begin{subequations}\label{HE}
\begin{align}
\label{HE1}
	& \mathbf E = \nabla u + \eta\\
\label{HE2}
	& \|\nabla\eta\|_{\Ld} = \|\nabla \times \mathbf E\|_{\Ld}\\
\label{HE3}
	& \max\left\{ \|\nabla u \|_{\Ld} \virg \| \eta\|_{\Ld} \right\} \le \|\mathbf E\|_{\Ld}\,.
\end{align}
\end{subequations}
Recalling equation \eqref{weak-faraday}, from \eqref{HE2} and \eqref{HE3} we deduce
\begin{equation}
\label{REG1}
	\| \eta\|_{H^1(\Omega;\mathbb R^3)} \le \|\mathbf E\|_{\Ld}  + \|\mu\partial_t\mathbf H\|_{\Ld} + \|\Jm\|_{\Ld} \,.
\end{equation}

By Sobolev embedding Theorem, the inclusion of $H^1(\Omega\mathbin{;}\mathbb R^3)$ into $ L^6(\Omega;\mathbb R^3) $ is continuous, and so is
the embedding of $L^6(\Omega;\mathbb R^3) $ into Morrey's space $L^{2,2}(\Omega;\mathbb R^3)$, thanks to H\"older inequality. Thus, 
\(
	\|\eta\|_{L^{2,2}(\Omega;\mathbb R^3) }\le C	\|\eta\|_{H^1(\Omega\mathbin{;}\mathbb R^3)}
\)
for a constant $C>0$ that depends on $r$, only.
Hence, by \eqref{REG1} we get
\begin{equation}
\label{REG2}
	\|\eta\|_{L^{2,2}(\Omega;\mathbb R^3) } \le C\Big[\|\mathbf E\|_{\Ld}  + \|\mu\partial_t\mathbf H\|_{\Ld} + \|\Jm\|_{\Ld} \Big]\,.
\end{equation}

Next, we pick $w\in H^1(\Omega)$ and we test equation \eqref{weak-ampere} with $\varphi = \nabla w$. By \eqref{HE1}, we obtain
\[
\int_\Omega\sigma\nabla u\cdot\nabla w\,dx = -\int_\Omega(\sigma\eta+\Je)\cdot \nabla w\,dx\,.
\]
By~\cite[Theorem 2.19]{T} with $\Gamma=\partial\Omega$ (see also Lemma 2.18 therein),
there exists $\bar\lambda\in(1,2] $, depending only on $\Lambda$, such that for all $\lambda\in(1,\bar	\lambda]$ we have
\begin{equation*}
	\|\nabla u \|_{L^{2,\lambda}(\Omega\mathbin{;}\mathbb R^3)} \le C \Big[ \|\nabla u\|_{\Ld} + \|\sigma\eta+\Je\|_{L^{2,\lambda}(\Omega\mathbin{;}\mathbb R^3)} \Big]\,,
\end{equation*}
for a suitable $C>0$, depending on $\Lambda$ and on $r$, only. By \eqref{1.2ii} and \eqref{HE3}, the latter implies
\begin{equation}
\label{REG3}
	\|\nabla u \|_{L^{2,\lambda}(\Omega\mathbin{;}\mathbb R^3)} \le C \Big[ \|\mathbf E\|_{\Ld} + \|\eta\|_{L^{2,\lambda}(\Omega\mathbin{;}\mathbb R^3)} + \|\Je\|_{L^{2,\lambda}(\Omega\mathbin{;}\mathbb R^3)} \Big]\,.
\end{equation}

Fix $\lambda\in(1,\bar\lambda]$. By \eqref{HE1}, \eqref{REG2}, and \eqref{REG3}, there exists $C>0$, depending only on $\Lambda$ and  $r$, with
\begin{equation}
\label{REG5}
	\|\mathbf E\|_{L^{2,\lambda}(\Omega)}\le C \Big[\|\mathbf E\|_{\Ld}+\|\mu\partial_t\mathbf H\|_{\Ld}+\|\Jm\|_{\Ld}+\|\Je\|_{L^{2,\lambda}(\Omega\mathbin{;}\mathbb R^3)} \Big]\,.
\end{equation}

We recall that $\mathbf H\in Y_\mu$. By Lemma~\ref{mettidiv}, this gives $\mathbf H\in H^1(\Omega\mathbin{;}\mathbb R^3)$.
In view of Remark~\ref{const:mu},
there exist $q\in H^1_0(\Omega)$,
and $\zeta\in \Huz$, with
\begin{equation}
\label{REGh0'}
\int_\Omega\zeta\cdot\nabla v\,dx = 0 \,,\qquad \text{for all }v\in H^1_0(\Omega)\,,
\end{equation}
such that  $\mathbf H = \nabla q + \zeta$ and
\begin{equation}
\label{REGh1'}
		\max\Big\{ \|\nabla q\|_{\Ld}\virg \| \zeta\|_{\Ld}\Big\} \le \|\mathbf H\|_{\Ld}\,.
\end{equation}
Then, by~\cite[Lemma 6]{A}, for a constant $C$ depending only on $\Lambda,r$ we have
\[
	\|\nabla\zeta\|_{L^{2,\lambda}(\Omega\mathbin{;}\mathbb R^{3\times3})} \le C
			\|\nabla\times\zeta\|_{L^{2,\lambda}(\Omega\mathbin{;}\mathbb R^3)}\,. 
\]
Thus, recalling  that $\nabla\times \zeta = \nabla \times\mathbf H$ and using equation \eqref{weak-ampere},  we arrive at
\begin{equation}
\label{REGh4'}
	\|\nabla\zeta\|_{L^{2,\lambda}(\Omega\mathbin{;}\mathbb R^{3\times3})} \le C\Big[ 
			\|\mathbf E\|_{L^{2,\lambda}(\Omega\mathbin{;}\mathbb R^3)} + \|\Je \|_{L^{2,\lambda}(\Omega\mathbin{;}\mathbb R^3)}
	\Big]
\end{equation}
where $C$ depends on $\Lambda,r$, only.

We note that, by \eqref{hpOmega}, there exists $\overline\rho_0>0$, depending only on $r$, such that
if $0<\rho<\overline\rho_0$ then 
\begin{enumerate}[(i)]
\item the boundary of $\Omega\cap B(x_0,\rho) $, in the sense of~\cite[Definition 3.2]{AL2002}, is {\em of Lipschitz class with constants $c\rho$, $L$},
with $c$ and $L$ depending on $r$, only;
\item $\Omega\cap B(x_0,\rho)$ satisfies the {\em scale-invariant fatness condition}, in the sense of~\cite[equation (2.3)]{AL2002}.
\end{enumerate}
Thus, by~\cite[Proposition 3.2]{AL2002}, for every $0<\rho<\overline\rho_0$
the following Poincar\'e inequality
\[
	\int_{B_\rho(x_0)\cap\Omega} \left\vert \zeta(x)-\intmed_{B_\rho(x_0)\cap\Omega} \zeta(y)\,dy\right\vert^2\!dx
	\le C\rho^2 \int_{B_\rho(x_0)\cap\Omega}|\nabla \zeta|^2\,dx\,,
\]
holds for all $x_0\in \Omega$, for a constant $C$ depending only on $r$. Hence,
\begin{equation}
\label{REGh5'}
\lbrack\zeta\rbrack_{\mathcal{L}^{2,\lambda+2}(\Omega\mathbin{;}\mathbb R^3)}\le C\, \lbrack\nabla\zeta\rbrack_{L^{2,\lambda}(\Omega\mathbin{;}\mathbb R^{3\times3})}\,.
\end{equation}
By
\eqref{REGh1'}, 
\eqref{REGh4'}, \eqref{REGh5'}, there exists a constant $C>0$ depending on
$\Lambda$ and $r$ such that
\begin{equation}
\label{REGh6}
		\|\zeta\|_{\mathcal{L}^{2,\lambda+2}(\Omega\mathbin{;}\mathbb R^3)}\le C \Big[
\|\mathbf H\|_{\Ld} +\|\mathbf E\|_{L^{2,\lambda}(\Omega\mathbin{;}\mathbb R^3)} + \|\Je \|_{L^{2,\lambda}(\Omega\mathbin{;}\mathbb R^3)}
		 \Big]\,.
\end{equation}

We recall that Campanato's space $\mathcal{L}^{2,\lambda+2}(\Omega)$, as a Banach space, is isomorphic to $C^{0,\alpha}(\overline\Omega)$, where 
$\alpha\in(0,\frac{1}{2})$ is given by $\alpha=(\lambda-1)/2$. Incidentally, we set $\alpha_0=(\bar\lambda-1)/2$, we observe that $\alpha\in(0,\alpha_0)$
and $\alpha_0\in(0,\frac{1}{2}]$,
because $\bar\lambda\in(1,2]$. 
Then, \eqref{REGh6} implies
\begin{equation}
\label{REGh7'}
			\|\zeta\|_{\mathcal{C}^{0,\alpha}(\overline\Omega\mathbin{;}\mathbb R^3)}\le C \Big[\|	\mathbf H\|_{\Ld}+
			\|\mathbf E\|_{L^{2,\lambda}(\Omega\mathbin{;}\mathbb R^3)} + \|\Je \|_{L^{2,\lambda}(\Omega\mathbin{;}\mathbb R^3)}
			\Big]\,.
\end{equation}

We take $w\in H^1_0(\Omega)$ and we test equation \eqref{weak-faraday} with $\psi=\nabla w$. By Fubini's Theorem and integrations by parts, we get
\[
 \int_\Omega\mu \mathbf H\cdot \nabla w\,dx - \int_\Omega\mu \mathbf H_0\cdot\nabla w\,dx
= \int_\Omega\int_0^t\Jm\cdot \nabla w\,ds\,dx\,.
\]
Since $\mathbf H =\nabla q+ \zeta$
and $w$ can be any element of $H^1_0(\Omega)$, it follows that $q\in H^1_0(\Omega)$ is a weak solution of the elliptic equation
\[
	\nabla\cdot(\mu\nabla q) = \nabla\cdot  \left( \int_0^t \Jm\,ds + \mu \mathbf H_0-\mu \zeta\right)\,.
\]
Then, classical global Schauder estimates (see, e.g.,~\cite[Theorem~2.19]{T} with $\Gamma=\varnothing$, and Lemma~2.18 therein) 
imply
\begin{equation}
\label{REGh8'}
	\|\nabla q\|_{C^{0,\alpha}(\overline\Omega\mathbin{;}\mathbb R^3)}\le 
C	\Big[ \int_0^t\|\Jm\|_{C^{0,\alpha}(\overline\Omega\mathbin{;}\mathbb R^3)}ds +\|\mu\mathbf H_0\|_{C^{0,\alpha}(\overline\Omega\mathbin{;}\mathbb R^3)}
	+\|\zeta\|_{C^{0,\alpha}(\overline\Omega\mathbin{;}\mathbb R^3)}\Big]\,,
\end{equation}
where the constant depends on $\Lambda$, on $r$.

Since $\mathbf H = \nabla q+ \zeta$, from \eqref{REGh7'} and \eqref{REGh8'} we deduce that
the estimate \eqref{holdestH} is valid for all $t$ that belong to the set $\mathcal{I}$ defined at the beginning of the proof. Since
$\mathcal{I}$ has full measure in $(0,T)$, clearly it follows that \eqref{holdestH} holds for a.e.\ $t\in(0,T)$.
\qed

\appendix\label{app}

\section{Helmoltz decompositions}

\subsection{Proof of Lemma~\ref{HelmE}}
We define
\[
	V = \left\{ u\in H^1(\Omega) \mathbin{\colon} \intmed_{\Omega} u \,dx =0\right\} 
\]
and we observe that $V$ is a closed subspace of the Hilbert space $H^1(\Omega)$.
By Poincar\'e's inequality and Lax-Milgram Lemma,
there exists a (unique) solution $u\in V$ to the variational problem
\begin{equation}
\label{HelmEvp0}
 \int_\Omega \nabla u\cdot\nabla v\,dx =\int_\Omega \mathbf F\cdot \nabla v\,dx\,,\quad \text{for all } v\in V\,.
\end{equation}
Since every $v\in H^1(\Omega)$ differs from some element of $V$ by a constant, from \eqref{HelmEvp0} we can infer 
\begin{equation}
\label{HelmEvp}
 \int_\Omega \nabla u\cdot\nabla v\,dx =\int_\Omega \mathbf F\cdot \nabla v\,dx\,,\quad \text{for all } v\in H^1(\Omega)\,.
\end{equation}
Setting $\eta = \mathbf F-\nabla u$, we have \eqref{HelmE1}  trivially, and \eqref{HelmEvp} implies \eqref{HelmE2}.
To conclude the proof, we test \eqref{HelmEvp} with $v=u$ and get
\begin{equation}
\label{HelmEvp1}
	\int_\Omega |\nabla u|^2\,dx = \int_\Omega \mathbf F\cdot \nabla u\,dx\,.
\end{equation}
Therefore, Cauchy-Schwartz inequality implies $\|\nabla u\|_{\Ld}\le \|\mathbf F\|_{\Ld} $. Then, we note that
\[
	\int_\Omega |\eta|^2 \,dx = \int_\Omega |\mathbf F|^2\,dx +\int_\Omega |\nabla u|^2\,dx - 2\int_\Omega \mathbf F\cdot \nabla u\,dx\,.
\]
Hence, recalling \eqref{HelmEvp1}, we have
$\|\eta\|_{\Ld}^2 \le \|\mathbf F\|_{\Ld}^2-\|\nabla u\|_{\Ld}^2\le \|\mathbf F\|_{\Ld}^2$ 
and we deduce \eqref{HelmE2.5}.

Now, we also assume that $\mathbf F\in \Hu$. Since $\nabla\times\eta = \nabla\times(\mathbf F-\nabla u)=\nabla\times \mathbf F$,
the (distributional) curl of $\eta$ belongs to
$\Ld$. Since \eqref{HelmE2} holds, in particular, for all $v\in H^1_0(\Omega)$, the (distributional) divergence $\nabla \cdot\eta$ of $\eta$ equals $0$. Moreover, again by \eqref{HelmE2},
 for every $v\in H^1(\Omega)$ 
\[
	\langle \gamma_{\partial\Omega}(v)\virg \eta\cdot n\rangle = \int_\Omega \nabla v\cdot\eta\,dx\,,
\]
where $\gamma $ is the trace operator from $H^1(\Omega)$ to $H^{\frac{1}{2}}(\partial\Omega)$ and $\langle\cdot\virg\cdot\rangle$ is the duality pairing between
$H^{-\frac{1}{2}}(\partial\Omega)$ and $ H^{\frac{1}{2}}(\partial\Omega)$. Hence $\eta\cdot n=0$ in $H^{-\frac{1}{2}}(\partial\Omega)$. Then, by an integration by parts,
we deduce that
$\eta\in H^1(\Omega\mathbin{;}\mathbb R^3)$ and 
$\|\nabla \eta \|_{\Ld} = \| \nabla \times \eta\|_{\Ld}$. 
Since $\nabla \times \eta = \nabla\times\mathbf F$, we conclude that $\|\nabla \eta \|_{\Ld} =\|\nabla\times\mathbf F\|_{\Ld}$ as desired.

\subsection{Proof of Lemma~\ref{lm:hh}}

Equation \eqref{hhnew0} with $ v=q$ reads as
\begin{equation}
\label{hhnew1}
\int_\Omega\mu |\nabla q|^2\,dx = \int_\Omega\mu\,\mathbf F\cdot\nabla q\,dx\,.
\end{equation}
Using Cauchy-Schwartz inequality and \eqref{1.2i}, from \eqref{hhnew1} we obtain $\|\nabla q\|_{\Ld}\le\Lambda \|\nabla \mathbf F\|_{\Ld}$, which gives the first inequality in \eqref{hh3};
setting $\zeta = \mathbf F-\nabla q$ and using \eqref{hhnew1} again we also get
\[
	\int_\Omega\mu|\zeta|^2\,dx {=} \!\!\int_\Omega\mu|\mathbf F|^2\,dx+\int_\Omega\mu |\nabla q|^2\,dx-2\int_\Omega\mu\,\mathbf F\cdot\nabla q\,dx  
	{=}\!\!\int_\Omega\mu|\mathbf F|^2\,dx-\int_\Omega\mu|\nabla q|^2\,dx\le \int_\Omega\mu|\mathbf F|^2\,dx\,,
\]
which gives the second inequality, too. Since $\zeta = \mathbf F-\nabla q$,
clearly \eqref{hh2} holds, $\zeta\in L^2(\Omega\mathbin{;}\mathbb R^3)$, and by \eqref{hhnew0} we also have $\zeta\in X_\mu$. 

If, in addition, $\mathbf F\in H^1(\Omega\mathbin{;}\mathbb R^3)$, then $\nabla\cdot\mathbf F\in \Ld(\Omega\mathbin{;}\mathbb R^3)$. Hence,
 by \eqref{hhnew0} and Elliptic Regularity we have $q\in H^2(\Omega)$ (see, e.g.,~\cite[\S 8.3]{GT}). By difference, $\zeta \in H^1(\Omega\mathbin{;}\mathbb R^3)$. Moreover, 
\begin{equation}
\label{hhnew3}
\begin{split}
	\int_\Omega \varphi\cdot\nabla\times\zeta\,dx-\int_\Omega \zeta\cdot\nabla\times\varphi\,dx & =
			\int_\Omega \varphi\cdot\nabla\times(\mathbf F-\nabla q)\,dx-\int_\Omega (\mathbf F-\nabla q)\cdot\nabla\times\varphi\,dx\\
			& = \int_\Omega \varphi\cdot\nabla\times\mathbf F\,dx-\int_\Omega \mathbf F\cdot\nabla\times\varphi\,dx
			+\int_\Omega \nabla q\cdot\nabla\times\varphi\,dx\,,
\end{split}
\end{equation}
for all given $\varphi\in C^1(\overline\Omega\mathbin{;}\mathbb R^3)$.
Now we also assume that
 $\mathbf F\times n=0$ in $H^{-\frac{1}{2}}(\partial\Omega\mathbin{;}\mathbb R^3)$. Then, by \eqref{2.2},
\begin{equation*}
\int_\Omega \varphi\cdot\nabla\times\mathbf F\,dx-\int_\Omega \mathbf F\cdot\nabla\times\varphi\,dx=0\,.
\end{equation*}
Since $q\in H^1_0(\Omega)$, by divergence theorem we also have
\begin{equation*}
	\int_\Omega \nabla q\cdot\nabla\times\varphi\,dx=0\,.
\end{equation*}
Inserting the last two identities in \eqref{hhnew3} we obtain
\[
	\int_\Omega \varphi\cdot\nabla\times\zeta\,dx-\int_\Omega \zeta\cdot\nabla\times\varphi\,dx=0\,.
\]
Since $\varphi$ was arbitrary, by \eqref{2.2} we deduce that $\zeta\times n=0$ in $H^{-\frac{1}{2}}(\partial\Omega\mathbin{;}\mathbb R^3)$.
Thus, $\zeta\in \Huz$. Recalling that $\zeta\in X_\mu$
and that by definition $Y_\mu = \Huz\cap X_\mu$, this concludes the proof.

\end{document}